\documentclass[12pt]{amsart}
\usepackage{latexsym}
\usepackage{amssymb, amsmath,mathrsfs,amsthm}
\usepackage{geometry}
\usepackage[OT2,T1]{fontenc}
\usepackage{setspace}
\usepackage{mathtools}

\usepackage{enumerate}
\usepackage{float, subfigure}
\usepackage[pagebackref,hypertexnames=false, colorlinks, citecolor=red, linkcolor=red]{hyperref}

\newtheorem{theorem}{Theorem}[section]
\newtheorem*{theorem*}{Theorem}
\newtheorem*{corollary*}{Corollary}
\newtheorem{lemma}[theorem]{Lemma}

\newtheorem{conjecture}[theorem]{Conjecture}
\newtheorem{corollary}[theorem]{Corollary}
\newtheorem{proposition}[theorem]{Proposition}
\usepackage{xcolor}

\DeclareSymbolFont{cyrletters}{OT2}{wncyr}{m}{n}
\DeclareMathSymbol{\Sha}{\mathalpha}{cyrletters}{"58}

\theoremstyle{remark}
\newtheorem{remark}[theorem]{Remark}
\newtheorem{example}[theorem]{Example}
\usepackage{chngcntr}
\counterwithin{figure}{section}
\numberwithin{equation}{section}

\begin{document}
\raggedbottom


\title[Two-Variable Compressions of Shifts]{Two-Variable Compressions of Shifts, Toeplitz Operators, and Numerical Ranges}
	\date{\today}
\author{Kelly Bickel$^\dagger$, Katie Quertermous, and Matina Trachana}

\address{Kelly Bickel, Department of Mathematics, Bucknell University, 380 Olin Science Building, Lewisburg, PA 17837, USA.}
\email{kelly.bickel@bucknell.edu}
\address{Katie Quertermous, Department of Mathematics and Statistics, James Madison University, 60 Bluestone Dr. MSC 1911, Harrisonburg, VA 22807, USA.}
\email{querteks@jmu.edu}

\address{Matina Trachana, Department of Mathematics and Informatics, University of Cagliari, Via Ospedale 72, 09124, Cagliari, Italy.}
\email{mtrachana@outlook.com}

\thanks{$\dagger$ Research supported in part by National Science Foundation DMS grant \#2000088.}

\maketitle
\begin{abstract}
This paper studies two-variable compressions of shifts associated to rational inner functions on the bidisk; these
generalize the classical compressions of the shift associated to finite Blasckhe products and are unitarily equivalent to one variable matrix-valued Toeplitz operators. This paper proves that a rational inner function is almost completely determined by these Toeplitz operator symbols but provides examples showing that (unlike in the one-variable case) rational inner functions are not determined by the numerical ranges of their compressed shifts. This paper also investigates related questions including methods of constructing these compressed-shift Toeplitz operators and when the associated numerical ranges are open and closed. 
\end{abstract}


\section{Introduction}

\subsection{One-variable setting} Consider the set of finite Blaschke products $B$, which are rational functions of the form
\[ B(z) = \lambda \prod_{j = 1}^n \frac{z - a_j}{1 - \overline{a_j} z},\]
where $|\lambda|=1$ and  $|a_j| <1$ for $j=1, \dots, n$. This class of functions acts roughly like polynomials on the unit disk $\mathbb{D} : = \{ z\in \mathbb{C}: |z|<1\}$ and plays a significant role in function theory on $\mathbb{D}$, for example regarding problems involving interpolation, approximation, and factorization, see \cite{gmr18, Garnett81}. They are also canonical examples of inner functions, namely holomorphic functions on $\mathbb{D}$ whose radial boundary values on the unit circle $\mathbb{T} :=\{ \tau \in \mathbb{C}: |\tau| =1\}$ have absolute value $1$ almost everywhere. In this paper, we are interested in associated operators called \emph{compressions of shifts}.

Specifically,  let $H^2(\mathbb{D})$ denote the classical Hardy space on the unit disk, namely the Hilbert space of holomorphic functions on $\mathbb{D}$ satisfying
\[ f(z) = \sum_{n=0}^{\infty} a_{n} z^n \text{ with }  \| f\|_{H^2(\mathbb{D})}^2  := \sum_{n=0}^\infty |a_{n}|^2 < \infty,\]
and let $S$ denote the shift operator on $H^2$ given by $[Sf](z) = z f(z)$. For each finite Blaschke product $B$ (and actually each inner function $\theta$), one can define the associated model space $K_B:=  H^2(\mathbb{D}) \ominus B H^2(\mathbb{D})$ and its compression of the shift
\[ S_B := P_B S |_{K_B},\]
where $P_{B}$ is the orthogonal projection from $H^2$ onto $K_B$. The study of compressions of the shift has received significant attention, partially due to the fact that these act as ``model operators'' for a large class of operators and partially because they satisfy some surprising but beautiful properties, see \cite{DGSV1, gmr16, gmr18, SBK10}. 

To see an example of such a property, recall that for any bounded operator $A$ on a Hilbert space $\mathcal{H}$, its numerical range $W(A)$ and numerical radius $w(A)$ are defined as follows: 
\[W(A) = \{\langle Ax, x\rangle: x \in \mathcal{H}, \|x\| = 1\} \ \text{ and } \  w(A) = \sup_{z\in W(A)} |z|.\]
The numerical range of an operator shows up in a variety of applications and both approximates the operator's spectrum and encodes additional information about the operator; for example an operator $A$ is self-adjoint if and only if $W(A) \subseteq \mathbb{R}$. See \cite{gr97} for additional properties and uses of numerical ranges.

For compressions of shifts, Gau and Wu in \cite{gw98} and Mirman in \cite{m98} showed that for each finite Blaschke product $B$ with $\deg B =n$, the numerical range $W(S_B)$ satisfies a beautiful geometric condition called the Poncelet property: for each point $\lambda \in \mathbb{T}$, there is an $(n+1)$-gon inscribed in $\mathbb{T}$ that circumscribes the boundary of $W(S_B)$ and has $\lambda$ as a vertex. This result and related techniques imply further properties of $W(S_B)$, for example that its closure is exactly the intersection of the numerical ranges of its unitary-$1$ dilations.

Further, because $\dim K_B = \deg B$, one can also view $S_B$ as a matrix. Specifically, the matrix representation of $S_B$ with respect to an orthonormal basis called the Takenaka-Malmquist-Walsh basis is exactly this $n\times n$, upper-triangular matrix, see \cite{gmr18, gw03}:
{\small \begin{equation} \label{eqn:SB} M_B = \left[\begin{array}{cccc}
a_1 & \sqrt{1 - |a_1|^2}\sqrt{1 - |a_2|^2} & \ldots & (\prod_{k = 2}^{n -
1}( -\overline{a_k})) \sqrt{1 - |a_1|^2}\sqrt{1 - |a_n|^2}\\ &&&\\ 0&a_2 &
\ldots & (\prod_{k = 3}^{n - 1}( -\overline{a_k})) \sqrt{1 -
|a_2|^2}\sqrt{1 - |a_n|^2}\\ &&&\\ \ldots & \ldots &\ldots &\ldots \\
&&&\\ 0 & 0 &0   & a_n
\end{array} \right].\end{equation}} 

From $\eqref{eqn:SB}$, we can see that $M_B$ only depends on the zeros of $B$, and the zeros of $B$ are exactly its eigenvalues. From this, one can deduce that if $B_1, B_2$ are finite Blaschke products, then $M_{B_1}$ will be unitarily similar to $M_{B_2}$ if and only if $B_1$ and $B_2$ have the same zero sets.  More generally, 
the following result encodes how uniquely the operator $S_B$ and its numerical range $W(S_B)$ depend on the choice of finite Blaschke product $B$:

\begin{theorem} \label{thm:unique1} Let $B_1$, $B_2$ be finite Blaschke products. Then the following are equivalent:
\begin{itemize}
\item[(i)]  $B_1= \lambda B_2$ for some $\lambda \in \mathbb{T}$;
\item[(ii)]  $S_{B_1}$ and $S_{B_2}$ are unitarily equivalent;
\item[(iii)] $W(S_{B_1}) = W(S_{B_2})$;
\item[(iv)] $\deg B_1 = \deg B_2$ and $W(S_{B_1}) \subseteq W(S_{B_2})$.
\end{itemize}
\end{theorem}
The equivalence of $(i)$ with $(ii)$ follows immediately from the matrix representation of $S_B$ in \eqref{eqn:SB}. The more subtle equivalencies of $(i)$ with $(iii)$ and $(iv)$ are due to Gau 
and Wu, see \cite[Theorem 4.1 and Lemma 4.2]{gw98}.
In general, square matrices $M_1$ and $M_2$ of the same size can have the same numerical range without being unitarily similar to one another. Thus, this is demonstrating a special property of compressions of the shift.

\subsection{Two-variable setting} 

 In this paper, we will investigate compressions of shifts in the two-variable setting and explore analogues of these uniqueness results and properties of numerical ranges. Specifically, let $\mathbb{D}^2$ denote the unit bidisk given by
$\mathbb{D}^2 = \{z= (z_1, z_2) \in \mathbb{C}^2: |z_1|, |z_2|<1\}.$
The two-variable generalizations of finite Blashke products on $\mathbb{D}^2$ are called \emph{rational inner functions} (RIFs); these are rational functions that are holomorphic on $\mathbb{D}^2$ with radial boundary values on $\mathbb{T}^2$ with modulus $1$ almost everywhere. 
Here are two examples: 
\begin{equation} \label{eqn:ex1} \phi(z_1, z_2) =  z_1 z_2^2 \left( \frac{z_1- \frac{1}{2}}{1- \frac{z_1}{2}}\right)  \left( \frac{z_2- \frac{i}{3}}{1+ i\frac{z_2}{3}}\right)   \ \ \text{ and } \ \ \theta(z_1, z_2) = \frac{2z_1 z_2-z_1-z_2}{2-z_1-z_2}.\end{equation}
The first example $\phi$ in \eqref{eqn:ex1} shows that we can build some RIFs using one-variable finite Blaschke products. The second example $\theta$ shows that RIFs do not always decompose into nice linear factors like finite Blaschke products. Moreover, they can have singularities on the boundary of $\mathbb{D}^2$, as evidenced by the fact that $\theta$ does not exist at $(1,1)$.  Given a rational function or polynomial $f$, we use $\deg_j f$ to denote the largest power of $z_j$ in the formula for $f$ and $\deg f = (\deg_1 f, \deg _2 f)$. So in \eqref{eqn:ex1}, $\deg \phi = ( 2, 3)$. 

RIFs do possess some  guaranteed structure.
 Rudin and Stout \cite{Rud69, RS65}  showed that for each RIF $\theta$, there is a  polynomial $p$ with no zeros in $\mathbb{D}^2$, an ordered pair of nonnegative integers $(m, n)$ with $(m,n) \geq {\rm{deg}} \,  p,$ and $\lambda \in \mathbb{T}$ such that 
\begin{equation} \label{eqn:RIF}  \theta(z_1, z_2) = \lambda\frac{ \tilde{p}(z_1, z_2)}{p(z_1, z_2)},\end{equation}
where  $\tilde{p}(z)$  is a reflection polynomial of $p$ defined by 
\[\tilde{p}(z_1, z_2) = z_1^m z_2^n \overline{ p\left (\tfrac{1}{\overline{z}_1}, \tfrac{1}{\overline{z}_2}\right)}.\]  Moreover, any function of the form \eqref{eqn:RIF} with $\tilde{p}$ a polynomial is an RIF.  Note that we have chosen to incorporate the full values of $m$ and $n$ into the definition of the reflection $\tilde{p}$  as was done in \cite{bg17} to simplify notation. 

The polynomial $p$ can actually be chosen uniquely (up to a constant multiple) if we further require that $p$ be \emph{atoral}, which in two variables means that $p$ has at most finitely many zeros on $\mathbb{T}^2$, see \cite[Theorem 3.3 and Section $4$]{ams06}. Since we have assumed that $p$ has no zeros in $\mathbb{D}^2$, the condition that $p$ be atoral is equivalent to  requiring that $p$ and $\tilde{p}$ have no common factors \cite{ams06, k15} and implies that $p$ has no zeros on the faces $\mathbb{D} \times \mathbb{T}$ and $\mathbb{T} \times \mathbb{D}$  of the bidisk, \cite[Theorem 4.9.1]{Rud69}.  Since $p$ has no zeros in $\mathbb{D}^2$, ${\rm deg} \,  \tilde{p} =(m,n) \geq {\rm deg} \, p$. Thus  $m$ will be the highest power of $z_1$ and $n$ will be the highest power of $z_2$ appearing in $\theta$, so $\deg \theta = (m,n).$  Throughout the paper, when we write $\theta=\lambda\frac{\tilde{p}}{p}$ is an RIF of degree $(m,n)$, we will be assuming that $p$ and $\tilde{p}$ are the atoral polynomial and reflection described above. Although RIFs are more complicated than finite Blaschke products, they still play similar roles in both interpolation and approximation applications on the bidisk, see \cite{am01, Rud69}.

Moving towards a definition of two-variable compressions of shifts, let $H^2(\mathbb{D}^2)$ denote the two-variable Hardy space on $\mathbb{D}^2$, namely the Hilbert space of functions holomorphic on $\mathbb{D}^2$ of the form 
\[ f(z) = \sum_{m,n=0}^{\infty} a_{m,n} z_1^mz_2^n \text{ with }  \| f\|_{H^2(\mathbb{D}^2)}^2 := \sum_{m,n=0}^\infty |a_{m,n}|^2 < \infty.\]
As in the one-variable case, we fix an RIF $\theta$ and define an associated two-variable model space by 
\[ \mathcal{K}_\theta :=H^2(\mathbb{D}^2) \ominus \theta H^2(\mathbb{D}^2).\] 
In this setting, there are now two shift operators $S_{z_1}, S_{z_2}$ defined on $H^2(\mathbb{D}^2)$ by $[S_{z_j}(f)](z) = z_j f(z)$ for $j=1,2$ and thus, two compressions of the shift
\[  \hat{S}^j_\theta := P_\theta S_{z_j} |_{\mathcal{K}_\theta},\]
on each $\mathcal{K}_\theta$, where $P_\theta$ is the orthogonal projection from $H^2(\mathbb{D}^2)$ onto $\mathcal{K}_\theta$. In the context of numerical ranges, these operators are not particularly interesting because generically the closure of each $W(\hat{S}^j_\theta)$ is $\overline{\mathbb{D}},$ see \cite{bg17}. For a setting that more closely mirrors the one-variable situation, we need to further decompose $\mathcal{K}_\theta$. The key tool available is called an \emph{Agler decomposition}. As proved in \cite{bsv05}, each $\mathcal{K}_\theta$ decomposes as
\[\mathcal{K}_{\theta} = \mathcal{S}_1 \oplus \mathcal{S}_2,\] 
where each $ \mathcal{S}_j$ is a closed subspace of $\mathcal{K}_\theta$ (and hence of $H^2(\mathbb{D}^2)$) invariant under the shift $S_{z_j}$.  These decompositions give the following alternative compressions of shifts
 \begin{equation} \label{eqn:Stheta} S^1_\theta = P_{\mathcal{S}_2} S_{z_1} |_{\mathcal{S}_2}  \ \ \text{ and } \ \  S^2_\theta = P_{\mathcal{S}_1} S_{z_2} |_{\mathcal{S}_1}, \end{equation}
 where $ P_{\mathcal{S}_j} $ is the orthogonal projection from $H^2(\mathbb{D}^2)$ onto $\mathcal{S}_j$. The first author and P. Gorkin studied these two-variable compressions of the shifts and their numerical ranges in \cite{bg17}, but a number of questions remained open. One key result from that paper is the following, which appears as Theorem 3.2 in  \cite{bg17}:

\begin{theorem} \label{thm:M} Let $\theta$ be an RIF of degree $(m,n)$ with $m \geq 1$. Then
there is an $m \times m$ matrix-valued function $M^1_\theta$ with entries continuous on $\overline{\mathbb{D}}$ and rational in $\overline{z_2}$ such that $S^1_{\theta}$ is unitarily equivalent to the matrix-valued Toeplitz operator $T_{M^1_\theta}$ defined by
\[T_{M^1_\theta} \overrightarrow{f}  = P_{H_2^2(\mathbb{D})^m}(M^1_{\theta} \overrightarrow{f}) \quad \text{ for }  \overrightarrow{f} \in H^2_2(\mathbb{D})^m,\]
where $H^2_2(\mathbb{D})$ denotes the one-variable Hardy space with variable $z_2$  and $H^2_2(\mathbb{D})^m:=\bigoplus_{j = 1}^m H_2^2(\mathbb{D}).$
\end{theorem}

So, these two-variable compressions of shifts are unitarily equivalent to one-variable matrix-valued Toeplitz operators. This fact will play an important role in the current paper, and the symbols $M_\theta^1$ will play a role analogous to the matrices $M_B$ in \eqref{eqn:SB} from the one-variable situation. By restricting to $\mathcal{S}_1$ instead of $\mathcal{S}_2$, one can analogously obtain symbols $M_\theta^2$ associated to the compressed shifts $S^2_\theta$.

Before proceeding, one should note that in Theorem \ref{thm:M}, the operator $S^1_{\theta}$ depends upon the chosen decomposition of $\mathcal{K}_{\theta}$ into the $S_{z_1}$-invariant subspace $\mathcal{S}_1$ and the $S_{z_2}$-invariant subspace $\mathcal{S}_2$. In general, there are multiple ways of choosing such decompositions. Meanwhile, the formula for $M^1_\theta$ further depends on a choice of orthonormal basis for the subspace $\mathcal{S}_2 \ominus S_{z_2} \mathcal{S}_2$; this relationship is explained in Section \ref{sec:2var}. As in \cite{bg17}, we will often suppress these dependences and just write $M_\theta^1$, with the understanding that there are implicit dependences on the chosen decomposition and basis. However, at some points, these choices will be important, and then we will use the notation $M_{\theta, \beta_2}^1$ to represent the matrix-valued function coming from a particular decomposition $\mathcal{S}_1\oplus \mathcal{S}_2$ with associated orthonormal basis $\beta_2$ of $\mathcal{S}_2 \ominus S_{z_2} \mathcal{S}_2$.  We can use similar notation $M^2_{\theta, \beta_1}$ to emphasize the basis $\beta_1$ of $\mathcal{S}_1 \ominus z_1\mathcal{S}_1$ when needed.

Although it is more complicated than the one-variable matrix $M_B$, the matrix-valued function $M_\theta^1$ has a natural interpretation in terms of the basis vectors of $\mathcal{S}_2$ and can sometimes be computed exactly. For instance, in Example \ref{ex:Mtheta_2}, we consider
\[\theta(z_1,z_2)=  \frac{2z^2_1z_2-z^2_1-z_1z_2}{2-z_1-z_2}\]
and find a decomposition of $\mathcal{K}_\theta$ as $\mathcal{K}_\theta = \mathcal{S}_1 \oplus \mathcal{S}_2$ such that  
\[M^1_{\theta}=\begin{bmatrix}0 & 0 \\ \frac{\sqrt{2}(1-\overline{z}_2)}{2-\overline{z_2}} & \frac{1}{2-\overline{z}_2}\end{bmatrix}. \]
We give more details about ways of computing $M^1_\theta$ in Section \ref{sec:2var}.

In this paper, we are motivated by this relationship to one-variable matrix-valued Toeplitz operators and build on the work from \cite{bg17} to further explore the structure and properties of the two-variable compressions of the shift defined in \eqref{eqn:Stheta}. However, it is worth noting that other two-variable shifts, shift-invariant subspaces, and related objects have been extensively studied in the literature, see, for example, \cite{bkds22, BL17, bcs23, dy00, lyz23, zLYS22, y02}. 

\subsection{Overview of Paper and Key Results}
In Section \ref{sec:2var}, we provide more intuition and details about the two variable situation. This includes additional explanations about Agler decompositions, a constructive version of Theorem \ref{thm:M}, and an example illustrating how to find $M_\theta^1$ for a given degree $(2,1)$ RIF $\theta$.

In Section \ref{sec:1n}, we study RIFs $\theta$ with $\deg \theta = (1,n)$. In this case, Theorem \ref{thm:M} says that $M^1_\theta$ is a $1\times 1$ function and thus, $S_\theta^1$ is unitarily equivalent to a scalar-valued Toeplitz operator. Section \ref{sec:1n} explores the implications of that observation. First, in Proposition \ref{prop:n1}, we identify a formula for the symbol of the associated Toeplitz operator. We then use known results about Toeplitz operators to characterize the spectrum, numerical range, and numerical radius of the compressed shift $S_\theta^1$ and deduce a number of its additional properties. For example, we establish a uniqueness result in Proposition \ref{prop:n1B}, characterize when $W(S_\theta^1)$ is open in Remark \ref{rem:open}, and identify which functions $M^1_\theta$ can be associated to some compressed shift $S^1_\theta$ in Proposition \ref{prop:psi}. 

For general RIFs,  the matrix-valued function $M^1_\theta$  is the two-variable analogue of \eqref{eqn:SB}.  To facilitate further study of this matrix-valued function, we construct $M^1_{\theta}$ for several classes of examples of RIFs $\theta$ in Section \ref{sec:product}.  Specifically, we consider RIFs $\theta$ that are  finite products of simpler RIFs and extend an algorithm from  \cite{bg17} to build decompositions of $\mathcal{K}_{\theta}$ and the corresponding $M^i_{\theta}$ from decompositions and corresponding matrix-valued functions for the RIF factors of $\theta$.  After analyzing the possible decompositions for degree $(1,0)$ and degree $(1,1)$ RIFs, we apply this algorithm to finite products of RIFs of these degrees, including arbitrary degree $(m,0)$ RIFs and RIFs of the form $\theta(z)=B(z_1z_2)$ for  a single-variable finite Blaschke product $B$. In this latter case, $M^1_{\theta}$ can be built from the matrix $M_B$ in \eqref{eqn:SB}.

The one-variable results in Theorem \ref{thm:unique1} and the results in the degree $(1,n)$ case imply three natural questions. First, in Section \ref{sec:M}, we consider \textbf{Question 1:} \emph{If $\theta$ and $\phi$ are two RIFs, what conditions on the matrix-valued functions $M^j_\theta$ and  $M^j_\phi$ guarantee that $\theta =\lambda \phi$ for some $\lambda \in \mathbb{T}$?}  To that end, we first study when $\theta$ and $\phi$ differ by inner factors involving $z_2$ and then characterize that in terms of the structure of $M_\theta^1$ and $M_\phi^1$. This appears as Theorem \ref{thm:uniqueness} and has the following corollary, which answers Question 1:

\begin{corollary*} \ref{cor:Munique}. Let $\theta, \phi$ be RIFs on $\mathbb{D}^2$ and $M_\theta^j, M_\phi^j$ be defined as above for $j=1,2$. Then $\theta=\lambda \phi,$ for some $\lambda \in \mathbb{T}$ if and only if there are unitary-valued functions $U_1, U_2$ on $\mathbb{T}$ such that for each $\tau \in \mathbb{T}$ and $j=1,2,$ 
$M^j_\theta(\tau) = U_j(\tau) M^j_\phi(\tau) U_j(\tau)^*.$
\end{corollary*} 

This shows that for any decomposition of $\mathcal{K}_\theta$, the symbols $M_\theta^1, M_\theta^2$ corresponding to the two compressed shifts almost uniquely identify $\theta$. Then in Corollary \ref{cor:2decomps}, we apply Theorem \ref{thm:uniqueness} to the case where $\theta = \phi$ but the symbols $M^1_\theta$ and $\widetilde{M}_\theta^1$ originate from different decompositions of $\mathcal{K}_\theta$. Example \ref{ex:Mtheta_2} illustrates finding such $\mathcal{K}_\theta$ decompositions, computing $M_\theta^1,$ and studying various properties of the unitary-valued functions $U_1, U_2$, such as the fact that they cannot always be constant.  

In Section \ref{sec:W}, we investigate \textbf{Question 2:}  \emph{If $\theta$ and $\phi$ are two RIFs, what does the numerical range equality $W(S_\theta^j) =W(S_\phi^j)$ for $j=1,2$ tell us about the relationship between $\theta$ and $\phi$?} Unfortunately, this relationship appears to be much more complicated than in the one variable situation (where $B_1 = \lambda B_2$). While Section \ref{sec:W} also discusses some partial results and examples showing that the answer will not be straightforward, the core of the section is Example \ref{ex:nonuniquest}. This example shows that the following two classes of degree $(2,2)$ RIFs for $t \in(0,1)$:
\[ \phi_t(z) =  \left( \frac{z_1 z_2 -t}{1-t z_1z_2}\right)^2 \ \ \text{ and } \  \ \psi_{t(2-t)}(z) =  z_1z_2\left( \frac{z_1 z_2 - t(2-t)}{1-t(2-t) z_1z_2}\right) \]
have identical numerical ranges $W(S^j_{\phi_t}) = W(S^j_{\psi_{t(2-t)}})$ for $j=1,2.$ However, these functions do not appear to satisfy an obvious relationship.

Lastly, Section \ref{sec:open} considers \textbf{Question 3:} \emph{If $\theta$ is an RIF, when is $W(S_\theta^j)$ open and when is $W(S_\theta^j)$ closed?} In Conjecture \ref{conj:open}, we suggest that the degree $(1,n)$ result likely holds in the general case; namely, if $\theta$ is a product of finite Blaschke products in $z_1$ and $z_2$ separately, then each $W(S_\theta^j)$ is closed, but otherwise, each $W(S_\theta^j)$ is open. Proposition \ref{prop:open} and Example \ref{ex:open1} discuss the situation when $\theta$ is a product of such finite Blaschke products. Specifically, Proposition \ref{prop:open} shows that for such $\theta$, there is always at least one pair $W(S_\theta^1), W(S_\theta^2)$ that must be closed, but Example \ref{ex:open1} shows that handling the general case (all $M_\theta^j$ coming from all possible decompositions of $\mathcal{K}_\theta$) is quite challenging. Meanwhile, Theorem \ref{thm:open} establishes the following partial result about openness:

\begin{theorem*}  \ref{thm:open}. Let $\theta =\lambda \frac{\tilde{p}}{p}$ be an RIF on $\mathbb{D}^2$. If there is no single straight line that intersects the numerical ranges $W(M^1_{\theta}(\tau))$ for all $\tau \in \mathbb{T}$, then $W(S_\theta^1)$ must be open.
\end{theorem*}

Example \ref{ex:open2} illustrates this theorem and provides additional support for Conjecture \ref{conj:open}.


\section*{Acknowledgements} 

Special thanks to Pamela Gorkin and Meredith Sargent for many useful discussions and insights related to both the overall direction and specific questions addressed in this paper.  As this work was initiated at the conference \emph{Working Groups for Women in Operator Theory} hosted by the Lorentz Center in Leiden, the authors would also like to thank the Lorentz Center staff and the workshop organizers for creating such a stimulating and collaborative environment.


\section{Two-variable Details} \label{sec:2var}

In this section, we provide some background on Agler decompositions of RIFs and how to use them to determine the matrix-valued function $M_\theta^1$ from Theorem \ref{thm:M}. 
 Throughout, let $\theta=\lambda \frac{\tilde{p}}{p}$ be an RIF of degree $(m,n)$.  Then, as proved in \cite{bsv05} and further studied in  \cite{b12},
 the model space $\mathcal{K}_{\theta}$ decomposes as \begin{align*} \mathcal{K}_{\theta}=\mathcal{S}_1^{\max} \oplus \mathcal{S}_2^{\min} = \mathcal{S}_1^{\min} \oplus \mathcal{S}_2^{\max},\end{align*} where $\mathcal{S}_1^{\max}$ (respectively $\mathcal{S}_2^{\max}$) is the largest, necessarily closed, subspace of $\mathcal{K}_{\theta}$ that is invariant under $S_{z_1}$ (respectively $S_{z_2}$), $\mathcal{S}_2^{\min}=\mathcal{K}_{\theta} \ominus \mathcal{S}_1^{\max}$,  and $\mathcal{S}_1^{\min}=\mathcal{K}_{\theta} \ominus \mathcal{S}_2^{\max}$. Somewhat surprisingly,  each $\mathcal{S}_j^{\min}$ is invariant under $S_{z_j}$.  These decompositions are canonical in the sense that if $\mathcal{K}_{\theta}=\mathcal{S}_1 \oplus \mathcal{S}_2$ is any orthogonal decomposition of $\mathcal{K}_{\theta}$ in which each $\mathcal{S}_j$ is $S_{z_j}$-invariant (we often write this as $z_j$-invariant), then $\mathcal{S}_j^{\min} \subseteq \mathcal{S}_j \subseteq \mathcal{S}_j^{\max}$.  

We now fix such an orthogonal  decomposition $\mathcal{K}_{\theta} = \mathcal{S}_1 \oplus \mathcal{S}_2$ in which each $\mathcal{S}_j$ is a closed $S_{z_j}$-invariant subspace.  It is well known that $\mathcal{K}_{\theta}$ is a reproducing kernel Hilbert space with reproducing kernel \begin{align*} K_{\theta}(z, w):=\frac{1 - \theta(z)\overline{\theta(w)}}{(1-z_1\overline{w_1})(1-z_2\overline{w_2})}.\end{align*} 
As discussed in the proofs of \cite[Theorems 2.5 and 3.7]{b12}, there exists a pair of holomorphic, positive semi-definite kernels $K_1, K_2 : \mathbb{D}^2 \times \mathbb{D}^2 \rightarrow \mathbb{C}$ such that each $\mathcal{S}_j$ is a reproducing kernel Hilbert space with kernel  $\frac{K_j}{1-z_j\overline{w_j}}.$ We write this as
\begin{align}\label{eqn:Skernels} \mathcal{S}_1 = \mathcal{H} \left(\frac{K_1(z,w)}{1-z_1\overline{w_1}}\right) \quad \text{and} \quad \mathcal{S}_2=\mathcal{H}\left(\frac{K_2(z,w)}{1-z_2\overline{w_2}}\right).\end{align}
An explicit formula for $K_j$ is given by $K_j(z, w)=(1-z_j\overline{w_j})(P_{\mathcal{S}_j}\left(K_{\theta}(\cdot, w)\right)(z)$ and 
\begin{align*} K_{\theta}(z,w)=\frac{K_1(z, w)}{1-z_1\overline{w_1}} + \frac{K_2(z,w)}{1-z_2\overline{w_2}}. \end{align*} 
This implies that  the pair of kernels $(K_1, K_2)$ satisfies the equation
 \begin{align}\label{Agler} 1-\theta(z)\overline{\theta(w)}=(1-z_1\overline{w_1})K_2(z,w) + (1-z_2\overline{w_2})K_1(z,w)\end{align}
 for all $z, w \in \mathbb{D}^2$. More generally, equation \eqref{Agler} is called an \emph{Agler decomposition of $\theta$} and any holomorphic positive  semi-definite kernels $(K_1, K_2)$ that satisfy \eqref{Agler} are called \emph{Agler kernels of $\theta$}.  If we start with a pair of Agler kernels $(K_1, K_2)$ of $\theta$, then the spaces $\mathcal{S}_1$ and $\mathcal{S}_2$ defined by \eqref{eqn:Skernels} are necessarily contractively contained in $\mathcal{K}_{\theta}$ and each  $\mathcal{S}_j$ is $S_{z_j}$-invariant. 
 However, unless the Agler decomposition of $\theta$ is unique, one must also verify that $\mathcal{S}_1 \cap \mathcal{S}_2=\{0\}$ in order to apply a corollary of Aronszajn's Sum of Kernels Theorem to conclude that $\mathcal{S}_1$ and $\mathcal{S}_2$ are orthogonal subspaces that are closed in the $H^2(\mathbb{D}^2)$  norm and  $\mathcal{K}_{\theta}=\mathcal{S}_1 \oplus \mathcal{S}_2$. See, for example, \cite{ag89, bsv05,  b12,  k11A, k11B} for more information about Agler kernels and \cite{PR} for general facts about reproducing kernel Hilbert spaces.

Continuing with our fixed decomposition $\mathcal{K}_{\theta}=\mathcal{S}_1 \oplus \mathcal{S}_2$ and the corresponding kernels $K_1, K_2$ that satisfy \eqref{eqn:Skernels}, note that $\mathcal{H}(K_2)=\mathcal{S}_2 \ominus z_2\mathcal{S}_2$.  As indicated in \cite[Theorem 2.2]{bg20}, $\dim(\mathcal{S}_2 \ominus z_2\mathcal{S}_2)=m$, and, provided that $m \geq 1$,  there exists an orthonormal basis of $\mathcal{S}_2 \ominus z_2\mathcal{S}_2$ of the form \begin{align} \label{eqn:S2basis} \beta_{2}=\left\{\frac{q_1}{p},  \ldots, \frac{q_m}{p}\right\} \text{where all the }  \, q_i \, \text{are  polynomials with} \,  {\rm deg} \, q_i \leq (m-1, n). \end{align}   Recall that we can recover the kernel $K_2$ from this basis via the formula $K_2(z, w) = \sum_{i=1}^m \frac{q_i(z)\overline{q_i(w)}}{p(z)\overline{p(w)}}$.  
We can also  recover the full subspace $\mathcal{S}_2$ via the unitary operator $\mathcal{U}_{\beta_{2}}: H_2^2(\mathbb{D})^m \rightarrow \mathcal{S}_2$ given by \begin{align} \label{unitaryS2}\mathcal{U}_{\beta_2}\vec{f}:=\sum_{i=1}^m \frac{q_i}{p}f_i \quad \text{for} \quad \vec{f}=(f_1, \ldots, f_m) \in H^2_2(\mathbb{D})^m.\end{align}  As in Theorem \ref{thm:M}, $H^2_2(\mathbb{D})$ denotes the one-variable Hardy space with variable $z_2$  and $H^2_2(\mathbb{D})^m:=\bigoplus_{j = 1}^m H_2^2(\mathbb{D}).$  It is also worth noting that if $m=0$, then $\theta$ has a unique Agler decomposition, see \cite[Corollary 1.16]{k10}. In this case, setting $\mathcal{S}_1:=\mathcal{K}_{\theta}$ and $\mathcal{S}_2:=\{0\}$ gives one (and hence all) such decompositions.   
While we are primarily focused  on $\mathcal{S}_2$, a similar analysis can be done for $\mathcal{S}_1$ to establish the existence of an orthonormal basis of $\mathcal{H}(K_1)=\mathcal{S}_1 \ominus z_1 \mathcal{S}_1$ of the form 
\begin{align} \label{eqn:S1basis} \beta_1=\left\{\frac{r_1}{p}, \ldots, \frac{r_n}{p} \right\} \text{where all the }   \, r_i \, \text{are polynomials with} \, {\rm deg} \, r_i \leq (m, n-1),\end{align} which can be used to define a unitary operator $\mathcal{U}_{\beta_1}: H^2_1(\mathbb{D})^n \rightarrow \mathcal{S}_1$ in an analogous fashion to \eqref{unitaryS2}. As expected, $H^2_1(\mathbb{D})$ denotes the one-variable Hardy space with variable $z_1.$ 

If we start with any pair $(K_1, K_2)$  of Agler kernels  of $\theta$, as noted in \cite[Theorem 2.1]{bg20}, $\mathcal{H}(K_1)$ and $\mathcal{H}(K_2)$ are both finite-dimensional and there exist  polynomials $r_1, \ldots, r_s, q_1, \ldots, q_t$ satisfying  ${\rm deg} \, r_i \leq (m, n-1)$ and  ${\rm deg} \, q_j \leq (m-1, n)$ such that 
\begin{align} K_1(z,w ) = \sum_{i=1}^{s}\frac{r_i(z)\overline{r_i(w)}}{p(z)\overline{p(w)}} \, \,  \, \text{and} \, \, \,  K_2(z,w)=\sum_{j=1}^{t}\frac{q_j(z)\overline{q_j(w)}}{p(z)\overline{p(w)}}. \end{align} 
Proofs of these facts can be found in \cite{bsv05,  b12, k10}.  For additional information about Agler decompositions of RIFs and related objects, including refined decompositions, construction methods, and interesting examples, we refer the readers to \cite{cw99, gw04, k08, k15, Kum89}. 
If all elements of the set 
\[\left\{z_1^k\frac{r_i}{p}, z_2^{\ell}\frac{q_j}{p} \, : |, k, \ell \in \mathbb{N} \cup \{0\}, 1 \leq i \leq s, 1 \leq j \leq t\right\}\] 
are pairwise orthogonal in $H^2(\mathbb{D}^2)$, then one can show that $\mathcal{H} \left(\frac{K_1(z,w)}{1-z_1\overline{w_1}}\right) \cap \mathcal{H}\left(\frac{K_2(z,w)}{1-z_2\overline{w_2}}\right)=\{0\}$ so defining $\mathcal{S}_1$ and $\mathcal{S}_2$ by \eqref{eqn:Skernels}  yields an orthogonal decomposition $\mathcal{K}_{\theta}=\mathcal{S}_1 \oplus \mathcal{S}_2$ of the desired form.  It then follows that  $s=n$, $t=m$, $\beta_1 = \{\frac{r_1}{p}, \ldots, \frac{r_n}{p}\}$ is an orthonormal basis of  $\mathcal{S}_1 \ominus z_1\mathcal{S}_1$ of form \eqref{eqn:S1basis}, and  $\beta_2 = \{\frac{q_1}{p}, \ldots, \frac{q_m}{p}\}$ is an orthonormal basis of  $\mathcal{S}_2 \ominus z_2\mathcal{S}_2$ of form \eqref{eqn:S2basis}.  The normality of the vectors in $\beta_1$ and $\beta_2$ follows from the discussion in \cite[Section 5.3]{PR}. 

Returning to the general setting and our fixed decomposition $\mathcal{K}_{\theta}= \mathcal{S}_1 \oplus \mathcal{S}_2$ and the basis $\beta_2$ given by \eqref{eqn:S2basis}, we note that the subspace $\mathcal{S}_2$ is invariant under the backward shift $S^{1 \, *}_{\theta}=S_{z_1}^*|_{\mathcal{S}_2}$ by \cite[Proposition 3.4]{b12}.  Thus, $S_{z_1}^*\frac{q_i}{p}$ can be written in the form of the sum in \eqref{unitaryS2} in a unique way for each $\frac{q_i}{p} \in \beta_2$.  
We can now state a constructive version of Theorem \ref{thm:M} that summarizes the proof of \cite[Theorem 3.2]{bg17}:

\begin{theorem} \label{thm:Mconstructive} Suppose $\theta=\lambda \frac{\tilde{p}}{p}$ is an RIF of degree $(m,n)$ with $m \geq 1$ and the decomposition $\mathcal{K}_{\theta}=\mathcal{S}_1 \oplus \mathcal{S}_2$ and basis $\beta_2=\{\frac{q_1}{p}, \ldots, \frac{q_m}{p}\}$ are as described above.  For each $1 \leq i \leq m$, let $h_{i1}, \ldots, h_{im}$ be the one-variable functions in $H^2_2(\mathbb{D})$ that satisfy \begin{align*} S_{z_1}^*\left(\frac{q_i}{p}\right) = \sum_{j=1}^m \frac{q_j}{p}h_{ij}.\end{align*}  Then each function $h_{ij}$ is a rational function of $z_2$ that is continuous on $\overline{\mathbb{D}}$ and \begin{align*} S_{\theta}^1 = \mathcal{U}_{\beta_2}T_{M_{\theta}^1}\mathcal{U}_{\beta_2}^* \end{align*} where $T_{M_{\theta}^1}$ is the matrix-valued Toeplitz operator with symbol $M_{\theta}^1:=\left[\, \,  \overline{h_{ij}}\, \, \right]$ and $\mathcal{U}_{\beta_2}$ is the unitary operator defined by \eqref{unitaryS2}.
\end{theorem}

Throughout this paper, the notation $M^1_{\theta}$, or $M^1_{\theta, \beta_2}$ when we wish to emphasize the choices of $\mathcal{S}_2$ and $\beta_2$, will always denote the matrix-valued function defined above.  If $m=0$, then Theorem \ref{thm:Mconstructive} does not apply. Instead, as discussed earlier, we must have $\mathcal{S}_2 =\{0\}$, which implies $S_{\theta}^1: \{0\} \rightarrow \{0\}$ and thus, $W(S_{\theta}^1) =\emptyset$.  Furthermore,  an analogous argument can be applied to $S^2_{\theta}$ if $n \geq 1$ using an orthonormal basis $\beta_1$ of $\mathcal{S}_1 \ominus z_1\mathcal{S}_1$ of form \eqref{eqn:S1basis}.  In this case, $S^2_{\theta}$ is unitarily equivalent via the unitary operator $\mathcal{U}_{\beta_1}$ to the matrix-valued Topelitz operator $T_{M_{\theta}^2}$ with symbol $M^2_{\theta}=M^2_{\theta, \beta_1}:=[\, \, \overline{\tilde{h}_{ij}} \, \, ]$ where for all $1 \leq i \leq n$, $\tilde{h}_{i1}, \ldots, \tilde{h}_{in} \in H^2_{1}(\mathbb{D})$ satisfy $S_{z_2}^*\left(\frac{r_i}{p}\right) = \sum_{j=1}^n \frac{r_j}{p}\tilde{h}_{ij}$.

We can connect the matrix-valued function $M^{1}_{\theta}$ to the one-variable case using slice functions.  Since $p$ has at most finitely many zeros on $\mathbb{T}^2$, the exceptional set for $\theta$, which is defined as   \begin{align} E_{\theta} :=\left\{\tau \in \mathbb{T} \, : \, p(t, \tau) = 0 \, \,  \text{for some} \, \,   t \in \mathbb{T}\right\}, \end{align}  is guaranteed to be finite.  Now fix an arbitrary $\tau \in \mathbb{T} \setminus E_{\theta}$.  The slice function $\theta_{\tau}$ defined by $\theta_\tau := \theta( \cdot, \tau)$ is a finite Blaschke product in the variable $z_1$  of degree $m$ and has a corresponding one-variable model space $\mathcal{K}_{\theta_{\tau}}$. By  \cite[Theorem 2.2]{bg17},  the restriction map $\mathcal{J}_{\tau}: \mathcal{H}(K_2) \rightarrow \mathcal{K}_{\theta_{\tau}}$ defined by $\mathcal{J}_{\tau}f:=f(\cdot, \tau)$ is a unitary operator, and thus  $\beta_{2, \tau} :=\{\mathcal{J}_{\tau}(\frac{q_i}{p}) \, : \, \frac{q_i}{p} \in \beta_2\} = \{\frac{q_1(\cdot, \tau)}{p(\cdot, \tau)}, \ldots, \frac{q_m(\cdot, \tau)}{p(\cdot, \tau)}\}$ is an orthonormal basis for $\mathcal{K}_{\theta_{\tau}}$. 
By the arguments in \cite[Theorems $3.2$ and $3.5$]{bg17}, the matrix $M_\theta(\tau)$ is the matrix representation of  the one variable compression of the shift $S_{\theta_\tau}$ with respect to the orthonormal basis $\beta_{2, \tau}$  since \[ 
\begin{aligned} 
\left \langle S_{\theta_\tau} \frac{q_j}{p}(\cdot, \tau),  \frac{q_i}{p}(\cdot, \tau) \right \rangle_{\mathcal{K}_{\theta_\tau}} &=  \left \langle  \frac{q_j}{p}(\cdot, \tau), S_{z_1}^* \frac{q_i}{p}(\cdot, \tau) \right \rangle_{\mathcal{K}_{\theta_\tau}}  \\
& = \sum_{k=1}^m \left \langle  \frac{q_j}{p}(\cdot, \tau), \frac{q_k}{p}(\cdot, \tau) h_{ik}(\tau) \right \rangle_{\mathcal{K}_{\theta_\tau}}  \\
&= \overline{h_{ij}(\tau)} = M^1_\theta(\tau)_{ij}.
\end{aligned}
 \]
Hence there exists a unitary matrix $U_{\tau}$ such that $M^1_{\theta}(\tau) = U_{\tau} M_{\theta_{\tau}}U_{\tau}^*$, where  $M_{\theta_{\tau}}$ is the matrix representation of $S_{\theta_{\tau}}$ with respect to the Takenaka-Malmquist-Walsh basis given by \eqref{eqn:SB}.  This shows that the eigenvalues of $M^1_{\theta}(\tau)$ are the zeros of $\theta_{\tau}$.

\begin{example} Let $p(z)=1-\frac{1}{2}z_1^2z_2$, and let $\theta$ be the degree $(2,1)$ RIF $\theta(z)=\frac{\widetilde{p}(z)}{p(z)}=\frac{z_1^2z_2-\frac{1}{2}}{1-\frac{1}{2}z_1^2z_2}$. Then straightforward computations show that 
\begin{align*} \frac{1-\theta(z)\overline{\theta(w)}}{(1-z_1\overline{w_1})(1-z_2\overline{w_2})}=\frac{r_1(z)\overline{r_1(w)}}{p(z)\overline{p(w)}(1-z_1\overline{w_1})} + \frac{q_1(z)\overline{q_1(w)}+q_2(z)\overline{q_2(w)}}{p(z)\overline{p(w)}(1-z_2\overline{w_2})}\end{align*}  for $r_1(z)=\frac{\sqrt{3}}{2}z_1^2$, $q_1(z)=\frac{\sqrt{3}}{2},$ and $q_2(z)=\frac{\sqrt{3}}{2}z_1$.  Therefore, \begin{align*} K_1(z,w)=\frac{r_1(z)\overline{r_1(w)}}{p(z)\overline{p(w)}} \, \, \text{and} \,  \,  K_2(z,w)=\frac{q_1(z)\overline{q_1(w)}}{p(z)\overline{p(w)}}+\frac{q_2(z)\overline{q_2(w)}}{p(z)\overline{p(w)}}
\end{align*} satisfy \eqref{Agler} and thus form a pair of Agler kernels for $\theta$.    The power series expansions at $0$ of $z_1^j\frac{r_1}{p}$, $z_2^k\frac{q_1}{p}$, and $z_2^{\ell}\frac{q_2}{p}$ have the forms \begin{align*} \sum_{i=0}^{\infty} a_iz_1^{2(i+1)+j}z_2^i,  \quad \sum_{i=k}^{\infty}b_iz_1^{2(i-k)}z_2^i,  \quad \text{and} \quad \sum_{i=\ell}^{\infty} c_iz_1^{2(i-\ell)+1}z_2^i,\end{align*} respectively, where the scalars $a_i, b_i$, and $c_i$ depend on the choices of $j$, $k$, and $\ell$.  Comparing powers, it is clear that the elements of $\{z_1^j\frac{r_1}{p}, z_2^k\frac{q_1}{p}, z_2^{\ell}\frac{q_2}{p} \, : \, j, k, \ell \in \mathbb{N} \cup \{0\}\}$ will be pairwise orthogonal in $H^2(\mathbb{D}^2)$ since they have no terms with matching pairs of powers.  Thus, we obtain the orthogonal decomposition $\mathcal{K}_{\theta}=\mathcal{S}_1 \oplus \mathcal{S}_2$ where \begin{align*} \mathcal{S}_1:=\mathcal{H}\left(\frac{r_1(z)\overline{r_1(w)}}{p(z)\overline{p(w)}(1-z_1\overline{w_1})}\right) \quad  and \quad \mathcal{S}_2:=\mathcal{H}\left(\frac{q_1(z)\overline{q_1(w)}+q_2(z)\overline{q_2(w)}}{p(z)\overline{p(w)}(1-z_2\overline{w_2})}\right)\end{align*} are closed subspaces that are $S_{z_1}$-invariant and $S_{z_2}$-invariant respectively and $\beta_1=\left\{\frac{r_1}{p}\right\}$ and $\beta_2=\left\{\frac{q_1}{p}, \frac{q_2}{p}\right\}$ are orthonormal bases of $\mathcal{S}_{1} \ominus z_1\mathcal{S}_1$ and $\mathcal{S}_2 \ominus z_2\mathcal{S}_2$ respectively.

To find $M^1_{\theta}$ with respect to $\beta_2$, first  recall that the adjoint of the shift $S_{z_1}^*$ is defined by 
\begin{align} \label{eqn:shiftformula} \big(S_{z_1}^*f \big)(z_1, z_2) = \frac{f(z_1, z_2) - f(0,z_2)}{z_1},\end{align}
for $f \in H^2(\mathbb{D}^2).$  Thus, \begin{align*} 
S_{z_1}^*\frac{q_1}{p} &= \frac{\frac{\sqrt{3}}{2(1-\frac{1}{2}z_1^2z_2)}-\frac{\sqrt{3}}{2(1-0)}}{z_1}= \frac{\sqrt{3}z_1z_2}{4(1-\frac{1}{2}z_1^2z_2)}=0\left(\frac{q_1}{p}\right)+\frac{1}{2}z_2\left(\frac{q_2}{p}\right)\, \, \text{and} \\
S_{z_1}^*\frac{q_2}{p} &=\frac{\frac{\sqrt{3}z_1}{2(1-\frac{1}{2}z_1^2z_2)}-\frac{0}{2(1-0)}}{z_1}=\frac{\sqrt{3}}{2(1-\frac{1}{2}z_1^2z_2)}=1\left(\frac{q_1}{p}\right) + 0 \left(\frac{q_2}{p}\right), \end{align*} 
so, by Theorem \ref{thm:Mconstructive}, $M^1_{\theta}(z_2)=\begin{bmatrix} 0 & \frac{1}{2}\overline{z_2} \\ 1 & 0 \end{bmatrix}$. It is easy to see that the exceptional set $E_{\theta}$ is empty, and  that, for any $\tau \in \mathbb{T}$, the eigenvalues of $M^1_{\theta}(\tau)$ are indeed the zeros of $\theta_{\tau}(z_1)=\frac{z_1^2\tau-\frac{1}{2}}{1-\frac{1}{2}z_1^2\tau}.$

\end{example}


\section{The Degree $(1,n)$ Case} \label{sec:1n}
In this section, we explore two-variable compressed shifts and their properties for RIFs with degree $(1,n)$.  First, Theorem \ref{thm:Mconstructive} has the following simple interpretation when $\theta$ is an RIF of degree $(1,n)$:
\begin{proposition} \label{prop:n1} 
Let $\theta =\lambda \frac{\tilde{p}}{p}$ be an RIF of degree $(1,n)$. Then the matrix $M^1_{\theta}$ from Theorem \ref{thm:Mconstructive} is given by
\begin{equation} \label{eqn:psi} \overline{M^1_\theta(z_2)}  = \frac{(-S^*_{z_1}p)(z_2)}{p(0, z_2)}.
\end{equation}
This formula for $M_\theta^1$ holds for any decomposition of $\mathcal{K}_\theta$ into $\mathcal{S}_1 \oplus \mathcal{S}_2$ and orthonormal basis $\beta_2$ of $\mathcal{S}_2 \ominus z_2 \mathcal{S}_2$.
 \end{proposition}

\begin{proof} Let $\{ \tfrac{q}{p}\}$ be a basis of $\mathcal{S}_2 \ominus z_2 \mathcal{S}_2$ with $\deg q \le (0,n)$.  Then applying \eqref{eqn:shiftformula}, \[\left(  S_{z_1}^*\tfrac{q}{p} \right)(z) = q(z_2)\left( \frac{ \tfrac{1}{p(z)} - \tfrac{1}{p(0, z_2)}}{z_1}\right)= \frac{q}{p}(z) \cdot \frac{-(S^*_{z_1}p)(z_2)}{p(0, z_2)}.\]
Thus, by Theorem  \ref{thm:Mconstructive}, $S_\theta^1$ is unitarily equivalent to the Toeplitz operator $T_{{M}^1_\theta}$, where $M^1_\theta$ is defined in \eqref{eqn:psi}.
\end{proof}

This allows us to immediately answer a number of questions about $S_\theta^1$ in the degree $(1,n)$ case.

\begin{corollary} \label{cor:n1nr} Let $\theta =\lambda \frac{\tilde{p}}{p}$ be an RIF of degree $(1,n)$, let $M^1_\theta$ be defined as in \eqref{eqn:psi}, and let $M^1_\theta(\mathbb{D})= \{ M^1_\theta(z): z\in \mathbb{D}\}$. Then the following hold:
\begin{itemize}
\item[i.] The spectrum of $S_\theta^1$ is  the closure of $M^1_\theta(\mathbb{D})$.
\item[ii.] The numerical range of $S_\theta^1$ is the convex hull of $M^1_\theta(\mathbb{D})$.
\item[iii.] The numerical radius of $S_\theta^1$ is the $\displaystyle \sup_{z\in \mathbb{D}} |M^1_\theta(z)|$.
\end{itemize}
\end{corollary}

\begin{proof} These follow immediately from known properties of co-analytic Toeplitz operators with bounded symbols, see, for example,  Corollary $2$ and the discussion before it in \cite{klein}.\end{proof} 

\begin{example} To illustrate Corollary \ref{cor:n1nr}, consider the following RIFs:
\begin{equation} \label{eqn:exphi} \theta(z_1, z_2) = \frac{2z_1 z_2-z_1-z_2}{2-z_1-z_2} \ \  \text{ and } \ \ \phi(z_1, z_2) = \frac{ 4z_1z_2^2-z_2^2-3z_1z_2-z_2+z_1}{4-z_1-3z_2-z_1z_2+z_2^2}.\end{equation}
A simple computation using \eqref{eqn:psi} reveals that 
\[ M^1_\theta(z_2) = \frac{1}{2-\bar{z}_2} \ \text{ and } M^1_\phi(z_2) = \frac{1+\bar{z}_2}{4-3\bar{z}_2+\bar{z}_2^2}.\]
The closures of $M^1_\theta(\mathbb{D})$ and $M^1_\phi(\mathbb{D})$ are both graphed below in Figure \ref{fig:Mtheta}.
\begin{figure}[ht] 
    \subfigure[Closure of $M^1_\theta(\mathbb{D})$]
      {\includegraphics[width=0.4 \textwidth]{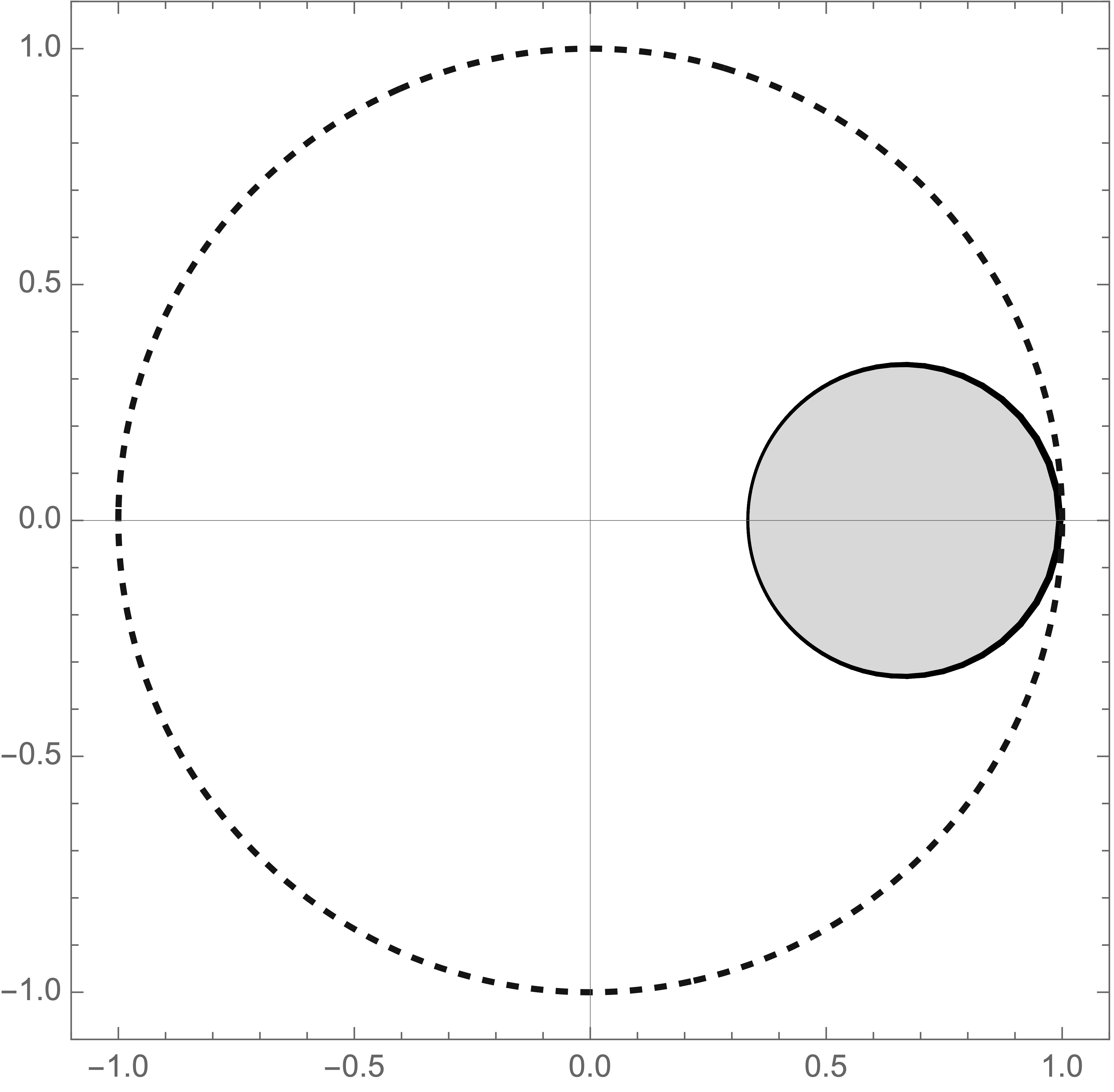}}
    \quad 
    \subfigure[Closure of $M^1_\phi(\mathbb{D})$]
      {\includegraphics[width=0.4 \textwidth]{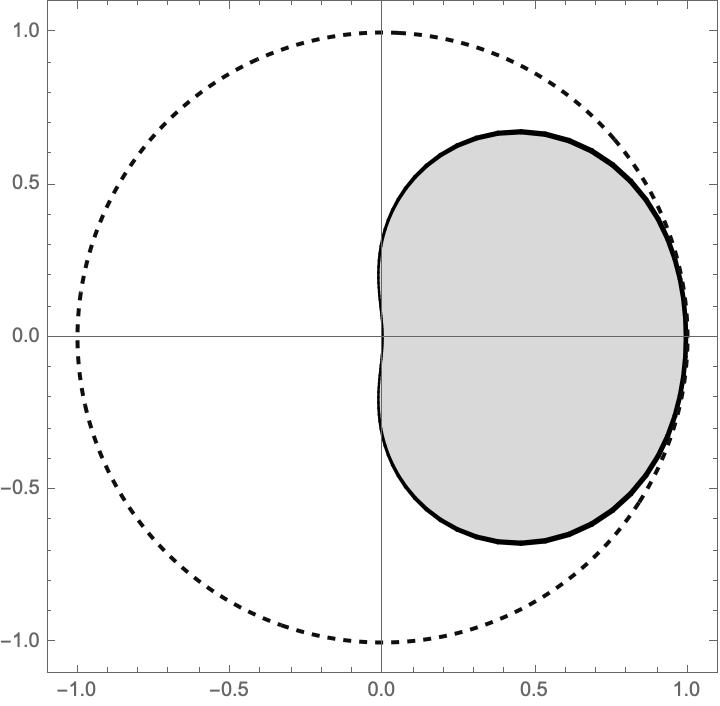}}
       \caption{\textsl{The closures of $M^1_\theta(\mathbb{D})$ and $M^1_\phi(\mathbb{D})$ in $\overline{\mathbb{D}}$ for $\theta$ and $\phi$ from \eqref{eqn:exphi}.}}\label{fig:Mtheta}
\end{figure}
 By Corollary \ref{cor:n1nr}, these plots align with the spectra of $S_\theta^1$ and $S_\phi^1$ and taking the convex hulls of their interiors gives the numerical ranges $W(S_\theta^1)$ and $W(S_\phi^1)$. The figures also indicate that the numerical radii satisfy $w(S_\theta^1)=1$ and $w(S_\phi^1)=1$. This aligns with \cite[Theorem 3.7]{bg17}, since both $\theta$ and $\phi$ have a singularity at $(1,1)$ on $\mathbb{T}^2$. One can also show that $M^1_\theta(\mathbb{D})$ (and hence $W(S_\theta^1)$) is the open Euclidean disk centered at $\tfrac{2}{3}$ with radius $\tfrac{1}{3}$. 
 
This observation about $W(S_\theta^1)$ for $\theta$ from \eqref{eqn:exphi} generalizes in the following way: \emph{each open Euclidean disk in $\mathbb{D}$ can be realized as  $W(S_\theta^1)$ for some RIF $\theta$ with degree $(1,1)$}. To see this, fix $r \in (0,1)$ and $c \in \mathbb{D}$ such that the open disk centered at $c$ with radius $r$, denoted here by $D_r(c)$, is inside $\mathbb{D}$. Define a polynomial $p(z) = 1 -\bar{c}z_1 - r z_1z_2$. Then $p$ does not vanish on $\mathbb{D}^2$ and it is easy to check that $p, \tilde{p}$ have no common factors. Define the RIF
\begin{equation} \label{eqn:cr} \theta(z) = \frac{\tilde{p}}{p}(z) = \frac{z_1z_2 -cz_2 -r}{1 -\bar{c}z_1 -r z_1 z_2}.\end{equation}
Then a short computation gives $M^1_\theta(z_2) =  c + r \bar{z}_2,$ which in turn implies that   
\[ W(S_\theta^1)= M^1_\theta(\mathbb{D})= D_r(c).\]
\end{example}

As discussed in the following remark, Corollary \ref{cor:n1nr} implies that the numerical range $W(S_\theta^1)$ is generally open for an RIF $\theta$ with degree $(1,n)$. 

\begin{remark} \label{rem:open} In the one-variable case, a compression of the shift $S_B$ associated to a finite Blaschke product $B$ acts on a finite-dimensional space. Hence,  the numerical range $W(S_B)$ is always closed. The two variable situation is very different.

Let $\theta =\lambda \frac{\tilde{p}}{p}$ be an RIF of degree $(1,n)$.  Corollary \ref{cor:n1nr} says that $W(S_\theta^1)$ is the convex hull of  $M^1_\theta(\mathbb{D}).$
When $M^1_\theta$ is constant,  $M^1_\theta(\mathbb{D})$ (and hence $W(S_\theta^1)$) is clearly closed. Note that $M^1_\theta$ is constant if and only if there is a polynomial $p_1$ and a constant $k$ with
\[ p(z_1, z_2) = p_1(z_2) + z_1 k p_1(z_2) = p_1(z_2) (1 + k z_1).\]
Since $\deg \theta = (1, n)$, $p$ also being of that form is equivalent to $\theta$ being the product of a  degree $1$ Blaschke product in $z_1$ and a degree $n$ Blaschke product in $z_2.$ However, once $\theta$ is not of that form (so $M^1_\theta$ is nonconstant), the open mapping theorem implies that $M^1_\theta(\mathbb{D})$ and hence, its convex hull $W(S_\theta^1)$, is open. 

Thus, the degree $(1,n)$ situation bifurcates into two cases: if  $\theta$ is the product of a degree $1$ Blaschke product in $z_1$ and a degree $n$ Blaschke product in $z_2$, then $W(S_\theta^1)$ is closed. Otherwise,  $W(S_\theta^1)$ is open. This answers Question $3$ for degree $(1,n)$ RIF $\theta$.
\end{remark}

We can also use the formula from Proposition \ref{prop:n1} to begin our investigation of Question $1$ in the setting of degree $(1,n)$ RIF $\theta$.

\begin{proposition} \label{prop:n1B} Let $\theta =\lambda_1 \frac{\tilde{p}}{p}$ and $\phi =\lambda_2 \frac{\tilde{q}}{q}$ be RIFs of degree $(1,n)$ and $(1, k)$  respectively. Then $M^1_\theta= M^1_\phi$  if and only if there exist finite Blaschke products $B_1$, $B_2$ such that $\theta(z) B_1(z_2) = \phi(z) B_2(z_2)$.
\end{proposition}

\begin{proof} By Proposition \ref{prop:n1}, the functions $M^1_\theta$ and $M^1_\phi$ are equal if and only if
\begin{equation} \label{eqn:Mpq} \frac{-(S^*_{z_1}p)(z_2)}{p(0, z_2)} =  \frac{-(S^*_{z_1}q)(z_2)}{q(0, z_2)}.\end{equation}

For the forward direction, assume \eqref{eqn:Mpq} and observe that it implies
\begin{equation} \label{eqn:qp} q(0, z_2) \left(p(0, z_2) +z_1(S^*_{z_1}p)(z_2)\right) = p(0, z_2) \left(q(0, z_2) +z_1(S^*_{z_1}q)(z_2)\right),\end{equation}
It's worth noting that \eqref{eqn:qp} implies that $p$ and $q$ have the same degree in $z_1$. Set  $r_1(z_2):=q(0,z_2)$ and $r_2(z_2):=p(0,z_2)$.  
Then \eqref{eqn:qp} is equivalent to
\begin{equation}   \label{eqn:qpr2} r_1(z_2) p(z) = r_2(z_2) q(z).\end{equation}
Since $r_1, r_2$ have no zeros on $\overline{\mathbb{D}}$, they are denominators  of finite Blaschke products. Set
\[ C_1(z_2) = \frac{\tilde{r}_1(z_2)}{r_1(z_2)} \ \ \text{ and } C_2(z_2) = \frac{\tilde{r}_2(z_2)}{r_2(z_2)},\]
where the $\tilde{}$ operation is with respect to the degree of each $r_j$ in $z_2.$
 Then \eqref{eqn:qpr2} implies that there are monomials $\gamma_1 z_2^{\ell_1}$ and $\gamma_2  z_2^{\ell_2}$ with $\gamma_1, \gamma_2 \in \mathbb{T}$ so that if we set 
 \[ B_1(z_2) = \gamma_1 z_2^{\ell_1}C_1(z_2)  \ \ \text{ and }  \ \ B_2(z_2) = \gamma_2 z_2^{\ell_2}C_2(z_2),\]
we obtain 
 \[ B_1(z_2)\theta(z)  =\gamma_1z_2^{\ell_1} \frac{\tilde{r}_1(z_2)\tilde{p}(z)}{r_1(z_2) p(z)} =  \gamma_2z_2^{\ell_2}\frac{\tilde{r}_2(z_2)\tilde{q}(z)}{r_2(z_2) q(z)} =B_2(z_2)\phi(z).\]
 
 For the other direction, assume $\theta(z) B_1(z_2) = \phi(z) B_2(z_2)$ for some finite Blaschke products $B_1(z_1)=\lambda_1 \frac{\tilde{r}_1(z_2)}{r_1(z_2)}$ and $B_2(z_2)= \lambda_2\frac{\tilde{r}_2(z_2)}{r_2(z_2)},$ where the $\tilde{}$ operation is with respect to $\deg(B_j)$ for each $j$ and $\lambda_1, \lambda_2 \in \mathbb{T}$. Then because $r_1p$ and $r_2q$ have no zeros on $\mathbb{D}^2$ and at mostly finitely many on $\mathbb{T}^2$, they are atoral. Since they are atoral denominator polynomials for the same RIF, there must be some nonzero constant $c$ with 
\[  r_1(z_2) p(z) = c r_2(z_2) q(z).\]  From this, one can rearrange terms to conclude that \eqref{eqn:Mpq} holds.
\end{proof} 

Lastly, we can use  Proposition \ref{prop:n1} to obtain the following characterization of the symbols $M^1_\theta$ associated to degree $(1,n)$ RIFs.  

\begin{proposition} \label{prop:psi} A one-variable scalar function $f(z_2)$ satisfies $\bar{f} = M^1_\theta$ for some RIF $\theta$ with $\deg_1( \theta) \ge 1$ if and only if $f$ satisfies the following three properties:
\begin{itemize}
\item[i.] $f$ is rational in $z_2$ and continuous on $\overline{\mathbb{D}}$,
\item[ii.] $|f|<1$ on $\mathbb{D}$,
\item[iii.] $f$ is not a finite Blaschke product.
\end{itemize}
 \end{proposition}

\begin{proof} 
For the forward direction, assume that $\overline{M^1_\theta} =f$ for some RIF $\theta = \lambda \frac{\tilde{p}}{p}$. By Theorem \ref{thm:Mconstructive}, $\theta$ necessarily has degree $(1,n)$ for some $n$ and $f$ is rational in $z_2$ and continuous on $\overline{\mathbb{D}}$. Writing $f =\frac{f_1}{f_2}$ with no common factors, \eqref{eqn:psi} implies that 
\[ \frac{f_1(z_2)}{f_2(z_2)}=  \frac{(-S^*_{z_1}p)(z_2)}{p(0, z_2)},\]
for all but at most finitely many $z_2$. 
Since $f_1$ and $f_2$ have no common factors, this implies that there is a polynomial $r$ so that
\[ -(S^*_{z_1}p)(z_2) = r(z_2) f_1(z_2) \ \ \text{ and } p(0, z_2) =r(z_2)  f_2(z_2).\]
Rearranging these equations gives $p(z) = r(z_2)\left( f_2(z_2)-z_1 f_1(z_2)\right).$
Since $p$ has no zeros on $(\mathbb{T} \times \mathbb{D}) \cup \mathbb{D}^2$, we must have $|f_1|<|f_2|$ on $\mathbb{D}$ and similarly, since $p$ has at most finitely many zeros on $\mathbb{T}^2$, $f$ cannot be a finite Blaschke product.

For the backwards direction, assume $f$ satisfies properties $(i)$-$(iii)$ and write $f = \frac{f_1}{f_2}$ with no common factors. If $f_1 \equiv 0$, set $\theta = z_1$, so $p\equiv1$. Otherwise, define
\[ p(z) = f_2(z_2) - z_1 f_1(z_2).\]
Then $p$ is an irreducible polynomial of degree $(1,n)$ for some $n$. By assumption (ii), $p$ has no zeros on $\mathbb{D}^2$.  We claim that $p$ is atoral. To see this, write
\[ \tilde{p}(z) = z_1 \tilde{f_2}(z_2) - \tilde{f_1}(z_2),\] 
where $\tilde{f}_2$ and $\tilde{f}_1$ are both computed using $\deg_2(p)$. By way of contradiction, assume that $p$ not atoral. Then 
$p$ and $\tilde{p}$ have a common factor. Since $p$ is irreducible, this implies
\[ z_1 \tilde{f_2}(z_2) - \tilde{f_1}(z_2) = \tilde{p}(z) = \lambda p(z) = \lambda \left( f_2(z_2) - z_1 f_1(z_2)\right), \]
for some $\lambda \in \mathbb{C}$. Since $|\tilde{p}(\tau)| = | p(\tau)|$ for all $ \tau \in \mathbb{T}^2$, we must have $\lambda \in \mathbb{T}$. Then $-\lambda f_1 = \tilde{f}_2$ and so $f = -\bar{\lambda} \frac{\tilde{f_2}}{f_2}$
is an RIF with $\deg_1(f) =0$ and hence, is a finite Blaschke product in $z_2$. Since that is not allowed, $p, \tilde{p}$ have no common factors and thus, $p$ is atoral.  Setting $\theta = \frac{\tilde{p}}{p}$ gives  $\overline{M^1_\theta} =f$.
\end{proof} 


\section{$M^1_{\Theta}$ for Products of Rational Inner Functions}  \label{sec:product}

To facilitate the study of compressions of the shift for more general RIFs, we now develop tools to allow us to compute the matrices $M^1_{\theta}$ for a variety of types of RIFs that will be key examples in our investigations.  Specifically,  we consider finite products of RIFs.  This work extends results of Bickel and Gorkin \cite{bg17} who studied the case in which  each factor is degree (1,1) and has a singularity on $\mathbb{T}^2$.

In the following, let $r \in \mathbb{N}$, $\Theta= \prod_{t=1}^r \theta_t$, where each $\theta_t=\lambda_t\frac{\tilde{p_t}}{p_t}$ is a degree $(m_t, n_t)$ RIF with corresponding orthogonal decomposition $\mathcal{K}_{\theta_t}=\mathcal{S}^{\theta_t}_1 \oplus \mathcal{S}^{\theta_t}_2$ such that $\mathcal{S}^{\theta_t}_1, \mathcal{S}_2^{\theta_t}$ are respectively $S_{z_1}$- and $S_{z_2}$-invariant.  For $1 \leq t \leq r$, assume that $m_t \geq 1$ and fix an orthonormal basis $\beta^{\theta_t}_2=\{f^t_1, \ldots, f^t_{m_t}\}$ for $\mathcal{S}^{\theta_t}_2 \ominus z_2\mathcal{S}^{\theta_t}_2.$ 

 To obtain a decomposition of $\mathcal{K}_{\Theta}$ as $\mathcal{S}_1 \oplus \mathcal{S}_2$, we notice that although Proposition 4.3 in \cite{bg17} was stated for degree $(1,1)$ functions $\theta_t$ each having a singularity on $\mathbb{T}^2$, most of its proof does not utilize those restrictions and thus establishes the following more general result.

\begin{proposition} \label{prop:productdecomp} 
In the above setting, $\displaystyle{\mathcal{K}_{\Theta} =\bigoplus_{t=1}^{r}\left(\prod_{k=1}^{t-1}\theta_k\right)\!\mathcal{K}_{\theta_t}}.$ For $i=1,2$, define \begin{align}\label{prodS2} \mathcal{S}_i:=\bigoplus_{t=1}^r \left(\left(\prod_{k=1}^{t-1} \theta_k\right)\!\mathcal{S}^{\theta_t}_i\right).\end{align}
Then $\mathcal{K}_{\Theta}=\mathcal{S}_1 \oplus \mathcal{S}_2$, $\mathcal{S}_1$ and $\mathcal{S}_2$ are respectively $S_{z_1}$- and $S_{z_2}$-invariant, and 
\begin{align}\label{prodbasis}\beta_2:=\left\{f_1^1, \ldots, f^1_{m_1}, \theta_1f^{2}_1, \ldots, \theta_1f^2_{m_2}, \ldots, \left(\prod_{k=1}^{r-1}\theta_k\right)\!f^r_1, \ldots,  \left(\prod_{k=1}^{r-1}\theta_k\right)\!f^r_{m_r}\right\}\end{align} is an orthonormal basis for
\begin{align*}\mathcal{S}_2 \ominus z_2\mathcal{S}_2 = \bigoplus_{t=1}^r \left(\prod_{k=1}^{t-1}\theta_k\right)\!\left(\mathcal{S}_2^{\theta_t} \ominus z_2\mathcal{S}^{\theta_t}_2\right).  
\end{align*}
\end{proposition} 

Note that if, for all $1 \leq t \leq r$, $n_t \geq 1$ and $\beta_1^{\theta_t}=\{f_1^{t, 1}, \ldots, f_{n_t}^{t, 1}\}$ is an orthonormal basis for $\mathcal{S}_1^{\theta_t} \ominus z_1\mathcal{S}_1^{\theta_t}$, then it also holds that  \begin{align}\label{prodS1} \beta_1:= \left\{f_1^{1,1}, \ldots, f^{1,1}_{n_1}, \theta_1f^{2,1}_1, \ldots, \theta_1f^{2,1}_{n_2}, \ldots, \left(\prod_{k=1}^{r-1}\theta_k\right)\!f^{r,1}_1, \ldots,  \left(\prod_{k=1}^{r-1}\theta_k\right)\!f^{r,1}_{n_r}\right\}\end{align}  is an orthonormal basis for $\mathcal{S}_{1} \ominus z_1 \mathcal{S}_1$ where $\mathcal{S}_1$ is defined by \eqref{prodS2}.  Here, we are not assuming that each $m_t \geq 1$.

We can now find an explicit formula for $M^1_{\Theta}$ as a lower-triangular block matrix that incorporates the matrices $M^1_{\theta_t}$ of the individual factors as its diagonal blocks.

\begin{theorem}\label{thm:productM} Let $\Theta= \prod_{t=1}^r \theta_t$ where each $\theta_t=\lambda_t\frac{\tilde{p_t}}{p_t}$ is a degree $(m_t, n_t)$ RIF with $m_t \geq 1$, and suppose $\mathcal{S}_2^{\theta_1}, \ldots, \mathcal{S}_2^{\theta_r}$, $\beta^{\theta_1}_2, \ldots, \beta^{\theta_r}_2$, $\mathcal{S}_2$, and $\beta_2$ are as described above. 
For $1 \leq t \leq r$, let $M_{\theta_t, \beta^{\theta_t}_2}^1$ be the matrix-valued function from Theorem \ref{thm:Mconstructive} for $\theta_{t}$  with respect to $\beta^{\theta_t}_2$, and let $g_1^t, \ldots, g_{m_t}^t \in H^2_2(\mathbb{D})$ such that $S_{z_1}^*\theta_t= \sum_{j=1}^{m_t} g_j^t f_{j}^t.$

Then the matrix-valued function $M^1_{\Theta}$ from Theorem \ref{thm:Mconstructive} with respect to the basis $\beta_2$ is a block matrix $[H_{st}]_{s, t = 1}^r$ where the $m_s \times m_t$ block $H_{st}$ is given by \begin{align*} H_{st} = \left\{\begin{array}{ll}  0_{m_s \times m_t}, & \text{if } s< t \\  M_{\theta_s, \beta_2^{\theta_s}}^1, &  \text{if } s=t\\\left(\prod_{\ell=t+1}^{s-1}\overline{\theta_{\ell}(0, z_2)}\right)\left[\, \, \overline{f_i^s(0, z_2)g_j^t(z_2)}\, \, \right], & \text{if } s>t\end{array} \right.\end{align*}

\end{theorem} 

\begin{proof}

By  Theorem \ref{thm:Mconstructive}, $M^1_{\Theta}$ is the block matrix $M^1_{\Theta}=[H_{st}]$ where the $m_s \times m_t$ blocks $H_{st}:=\left[\overline{h_{ij}^{st}}\right]$ are built from the unique functions $h^{st}_{ij} \in H^2_2(\mathbb{D})$ that satisfy 
\begin{align*} S_{z_1}^*\left(\left(\prod_{k=1}^{s-1}\theta_k\right)\!f_i^s\right)=\sum_{t=1}^r \sum_{j=1}^{m_t}h_{ij}^{st}\left(\prod_{k=1}^{t-1}\theta_k\right)\!f_j^t.\end{align*}
It is well-known that if $G, H \in H^2(\mathbb{D}^2)$ with $GH \in H^2(\mathbb{D}^2)$, then $S_{z_1}^*(GH)=GS_{z_1}^*(H)+H(0,z_2)S_{z_1}^*(G)$.  Applying this inductively,   we obtain 
 \begin{align*} S_{z_1}^*\left(\left(\prod_{k=1}^{s-1}\theta_k\right)\!f_i^s\right)=\left(\prod_{k=1}^{s-1}\theta_k\right)\!S_{z_1}^*(f_i^s) +  \sum_{t=1}^{s-1}\left(\prod_{k=1}^{t-1}\theta_k\right)\!\left[\left(\left(\prod_{\ell=t+1}^{s-1}\theta_{\ell}\right)\!f_i^s\right)\!(0, z_2)\right]\!S_{z_1}^*(\theta_t).\end{align*}
 
Straightforward computations show that  each $S_{z_1}^*(\theta_t) \in \mathcal{K}_{\theta_t}$ and $S_{z_1}^*(\theta_t)  \perp \mathcal{S}_{1}^{\theta_t}$ since $\mathcal{S}_{1}^{\theta_t}$ is a $S_{z_1}$-invariant subspace of $\mathcal{K}_{\theta_t}$, so $S_{z_1}^*(\theta_t) \in \mathcal{S}_{2}^{\theta_t}$. Thus, for each $t$, there exist functions $g_1^t, \ldots, g_{m_t}^t \in H_2^2(\mathbb{D})$ such that $S_{z_1}^*(\theta_t)=\sum_{j=1}^{m_t} g_j^t f_{j}^t.$ By construction, $M^1_{\theta_s, \beta^{\theta_s}_2}
=\left[\,  \overline{\hat{h}^s_{ij}} \, \right]$ where the functions  $\hat{h}^s_{ij} \in H^2_2(\mathbb{D})$ satisfy $S_{z_1}^*f_i^s = \sum_{j=1}^{m_s} \hat{h}^s_{ij}f_j^s$.  

Thus,  \begin{align*} S_{z_1}^*\left(\left(\prod_{k=1}^{s-1}\theta_k\right)\!f_i^s\right)& =\sum_{j=1}^{m_s}\hat{h}^s_{ij}\left(\prod_{k=1}^{s-1}\theta_k\right)f_j^s\\ & \quad + \sum_{t=1}^{s-1} \sum_{j=1}^{m_t} \left(\left(\left(\prod_{\ell=t+1}^{s-1}\theta_{\ell}\right)\!f_i^s\right)\!(0, z_2)\right)\!g_j^t(z_2)\left(\prod_{k=1}^{t-1}\theta_k\right)\!f_j^t.\end{align*}  Hence, for $1 \leq s, t \leq r,$ $1 \leq i \leq m_s$, and $1 \leq j \leq m_t$, the function \begin{align*} h_{ij}^{st}(z_2)=\left\{\begin{array}{ll} 0, & \text{if } s <t \\ \hat{h}^{s}_{ij}(z_2), & \text{if } s=t \\ 
\left(\prod_{\ell= t+1}^{s-1} \theta_{\ell}(0, z_2)\right)f_i^s(0, z_2)g_j^t(z_2), & \text{if } s >t\end{array}\right.,\end{align*} which yields the desired block matrix formula. \end{proof} 

Note that if $n_t \geq 1$ for $1 \leq t \leq r$, then  by applying an analogous argument using $S_{z_2}^*$, one can also obtain that for $\Theta$ described as above and $\beta_1$ given by \eqref{prodS1}, $M^2_{\Theta, \beta_1}=[\hat{H}_{st}]_{s,t=1}^r$ with $n_s \times n_t$ matrix blocks  \begin{align*} \hat{H}_{st} = \left\{\begin{array}{ll}  0_{n_s \times n_t}, & \text{if } s< t \\  M_{\theta_s, \beta_1^{\theta_s}}^2, &  \text{if } s=t\\\left(\prod_{\ell=t+1}^{s-1}\overline{\theta_{\ell}(z_1,0)}\right)\left[\overline{f_i^{s,1}(z_1, 0) g_j^{t,1}(z_1)}\right], & \text{if } s>t\end{array} \right.\end{align*} where, for $1 \leq t \leq r$, $g_{1}^{t,1}, \ldots, g_{n_t}^{t,1} \in H_1^2(\mathbb{D})$ satisfy $S_{z_2}^*\theta_t=\sum_{j=1}^{n_t} g_j^{t,1}f_{j}^{t,1}$.

Applying Theorem \ref{thm:productM} requires knowledge of the bases and corresponding matrix-valued functions associated to the building block functions $\theta_t$.  We can fully describe these for degree $(1,0)$ and degree $(1,1)$ RIFs $\theta$.  While the $1 \times 1$ matrix-valued functions $M^1_{\theta}$ are fully described by Proposition \ref{prop:n1} in this case, further investigations are required to specify the possible bases of $\mathcal{S}_2 \ominus z_2 \mathcal{S}_2$. The following work extends results in \cite{bg17} by providing an alternative approach to Example 3.3 and expanding Lemma 4.1, which considered degree $(1,1)$ RIFs with a singularity in $\mathbb{T}^2$.  For the degree (1,0) case, note that any degree $(m,0)$ RIF must be a finite Blaschke product in the variable $z_1$, so we can describe the general $(1,0)$ case via the following lemma.

\begin{lemma} \label{lemma:1_0} Suppose $a \in \mathbb{D}$, $\lambda \in \mathbb{T}$ and $\theta(z)=\lambda\frac{z_1-a}{1-\overline{a}z_1}$. If $\mathcal{K}_{\theta} = \mathcal{S}_1 \oplus \mathcal{S}_2$ is an orthogonal decomposition of $\mathcal{K}_{\theta}$ into closed subspaces that are respectively $S_{z_1}$- and $S_{z_2}$-invariant and $\beta_2$ is an orthonormal basis of  $\mathcal{S}_2 \ominus z_2\mathcal{S}_2$ of form \eqref{eqn:S2basis}, then $\mathcal{S}_1=\{0\}$, $\mathcal{S}_2=\mathcal{K}_{\theta},$ 
\begin{align*} M^1_{\theta} (z_2) =a,  \quad \quad  \beta_2=\left\{\frac{\delta \sqrt{1-|a|^2}}{1-\overline{a}z_1}\right\} \,\,  \text{for some} \, \,  \delta \in \mathbb{T}, \quad \text{and} \quad (S_{z_1}^* \theta)(z)=\lambda\frac{1-|a|^2}{1-\overline{a}z_1}.\end{align*}   

\end{lemma} 

\begin{proof}Since ${\rm deg} \, \theta = (1,0)$, we know from Section \ref{sec:2var}  that  $\dim(\mathcal{S}_2 \ominus z_2\mathcal{S}_2)=1$ and $\beta_2=\{\frac{q}{p}\}$ where $p(z)=1-\overline{a}z_1$ and $\deg q \leq (0,0)$, i.e.\ $q$ is a constant.  Norm computations show that $|q|$ must equal $\sqrt{|1|^2-|a|^2}$.  If we recover the reproducing kernel $K_2$ for  $\mathcal{S}_2 \ominus z_2 \mathcal{S}_2$ from $\beta_2$, $\frac{K_2}{1-z_2\overline{w_2}}=\mathcal{K}_{\theta}$, so $\mathcal{S}_2=\mathcal{K}_{\theta}$ and $\mathcal{S}_1=\{0\}$. The formula for $S_{z_1}^*\theta$ follows immediately from \eqref{eqn:shiftformula}. \end{proof} 

\begin{lemma} \label{lemma:1_1}Suppose $\theta=\lambda \frac{\tilde{p}}{p}$ is a degree $(1,1)$ RIF  with $p(z)=b+cz_1+dz_2+ez_1z_2$.  Let \begin{align*} A_1&:=(|b|^2-|e|^2)+(|c|^2-|d|^2), \quad A_2:=(|b|^2-|e|^2)-(|c|^2-|d|^2), \\ \gamma_1&:=\overline{b}c-\overline{d}e,  \quad \gamma_2:=\overline{b}d-\overline{c}e, \quad \zeta_i:=\begin{cases} \frac{\gamma_i}{|\gamma_i|}, & \gamma_i \neq 0 \\ 1, & \gamma_i=0 \end{cases} \, \, \text{for} \, \, i= 1, 2, \quad \text{and} \\ B&:=\sqrt{A_1^2-4|\gamma_1|^2}=\sqrt{A_2^2-4 |\gamma_2|^2}. \end{align*} Define the two-variable polynomials $r^+$, $r^-$, $q^+$, $q^-$ by \begin{align}\label{eqn:qr1_1} 
r^{\pm}(z) = \sqrt{\frac{A_1 \pm B}{2}} + \zeta_1\sqrt{\frac{A_1 \mp B}{2}}z_1 \quad \text{and} \quad q^{\pm}(z)=\sqrt{\frac{A_2 \pm B}{2}} + \zeta_2\sqrt{\frac{A_2 \mp B}{2}}z_2. 
\end{align} 
Then the only orthogonal decompositions $\mathcal{K}_{\theta} = \mathcal{S}_1 \oplus \mathcal{S}_2$ of $\mathcal{K}_{\theta}$ into closed subspaces that are respectively $S_{z_1}$- and $S_{z_2}$-invariant  are   $\mathcal{K}_{\theta}=\mathcal{H}\left(\frac{r^-(z)\overline{r^-(w)}}{p(z)\overline{p(w)}(1-z_1\overline{w_1})}\right) \oplus \mathcal{H}\left(\frac{q^+(z)\overline{q^+(w)}}{p(z)\overline{p(w)}(1-z_2\overline{w_2})}\right)$   and $\mathcal{K}_{\theta}=\mathcal{H}\left(\frac{r^+(z)\overline{r^+(w)}}{p(z)\overline{p(w)}(1-z_1\overline{w_1})}\right) \oplus  \mathcal{H}\left(\frac{q^-(z)\overline{q^-(w)}}{p(z)\overline{p(w)}(1-z_2\overline{w_2})}\right),$
 and these two decompositions are the same if and only if $B=0,$ which is equivalent to $\theta$ having a singularity on $\mathbb{T}^2$.  
The corresponding pairs of orthonormal bases for $\mathcal{S}_1 \ominus z_1\mathcal{S}_1$ and   $\mathcal{S}_2 \ominus z_2\mathcal{S}_2$ of forms \eqref{eqn:S1basis} and \eqref{eqn:S2basis} are respectively 
 \begin{align*} \beta_1^-=\left\{\delta_1\frac{r^-}{p}\right\}, \,  \beta_2^+ = \left\{\delta_2 \frac{q^+}{p}\right\} \qquad \text{and} \qquad  \beta_1^+=\left\{\delta_1 \frac{r^+}{p}\right\}, \, \beta_2^-=\left\{\delta_2\frac{q^-}{p}\right\}\end{align*}  for some $\delta_1, \delta_2 \in \mathbb{T}$, and in both cases \begin{align*}M_{\theta}^1(z_2)=\frac{-\overline{ez_2} -\overline{c}}{\overline{dz_2}+\overline{b}}  \qquad \text{and} \qquad  M_{\theta}^2(z_1)=\frac{-\overline{ez_1}-\overline{d}}{\overline{cz_1}+\overline{b}}. \end{align*} Moreover, \begin{align*}(S_{z_1}^*\theta)(z)=\lambda \overline{\zeta_2}\frac{q^+(z)q^-(z)}{p(z)p(0,z_2)} \qquad \text{and} \qquad (S_{z_2}^*\theta)(z)=\lambda \overline{\zeta_1}\frac{r^+(z)r^-(z)}{p(z)p(z_1, 0)} .\end{align*}
\end{lemma} 

\begin{proof} Suppose $\mathcal{K}_{\theta}=\mathcal{S}_1 \oplus \mathcal{S}_2$ is an orthogonal decomposition of $\mathcal{K}_{\theta}$ into closed subspaces that are respectively $S_{z_1}$- and $S_{z_2}$-invariant.  By the discussion in Section \ref{sec:2var}, $\mathcal{S}_1 \ominus z_1\mathcal{S}_1$ and $\mathcal{S}_2 \ominus z_2\mathcal{S}_2$ are both dimension one with corresponding orthonormal bases $\beta_1=\{\frac{r}{p}\}$ and $\beta_2 = \{\frac{q}{p}\}$ where ${\rm deg} \,  r \leq (1, 0)$ and ${\rm deg} \, q \leq (0,1)$.   Since $K_1(z,w)=r(z)\overline{r(w)}/p(z)\overline{p(w)}$ and $K_2(z,w)=q(z)\overline{q(w)}/p(z)\overline{p(w)}$ must satisfy \eqref{Agler}, we have that \begin{align*}|p(z)|^2 - |\tilde{p}(z)|^2=(1-|z_1|^2)|q(z)|^2+(1-|z_2|^2)|r(z)|^2\end{align*} for all $z \in \mathbb{D}^2$ and thus all $z \in \overline{\mathbb{D}}^2$ by continuity.  It follows that, for all $\eta \in \mathbb{T}$, \begin{align}\label{eqn:basisequality} |r(\eta, 0)|^2=|b+c\eta|^2-|d+e\eta|^2 \quad \text{and} \quad |q(0, \eta)|^2=|b+d\eta|^2-|c+e\eta|^2.\end{align}  From this we can deduce that $r$ must be a unimodular constant multiple of $r^+$ or $r^{-}$ and $q$ must be a unimodular constant multiple of $q^+$ or $q^-$. We show here the computations for $q$ and leave the similar computations for $r$ to the reader.

Suppose $q(z)=t+sz_2$.  Then expanding the second equation in \eqref{eqn:basisequality}, we obtain that $|t|^2+|s|^2 +  2{\rm Re}(\overline{t}s\eta)=A_2+2{\rm Re}(\gamma_2\eta)$ for all $\eta \in \mathbb{T}$.  Thus, \begin{align} \label{eqn:S21_1} |t|^2+|s|^2=A_2 \quad \text{and} \quad \overline{t}s=\gamma_2.\end{align}  If $\gamma_2=0$, this immediately implies that  $q(z)=\delta_2\sqrt{A_2}=\delta_2q^+(z)$ or $q(z)=\delta_2\sqrt{A_2}z_2=\delta_2q^-(z)$ for some $\delta_2 \in \mathbb{T}$.  Now suppose $\gamma_2 \neq 0$.  Then $s=\frac{\gamma_2}{\bar{t}}$, so by substituting into the first equation in \eqref{eqn:S21_1}, we obtain $|t|^4-A_2|t|^2+|\gamma_2|^2 =0$ and thus $|t|^2=\frac{A_2 \pm \sqrt{A^2_2-4|\gamma_2|^2}}{2}=\frac{A_2 \pm B}{2}$.  Hence \begin{align*} q(z) = \delta_2\sqrt{\frac{A_2 \pm B}{2}}+ \frac{\gamma_2}{\overline{\delta_2}\sqrt{\frac{A_2 \pm B}{2}}}z_2 = \delta_2 \left(\sqrt{\frac{A_2 \pm B}{2}} + \frac{\gamma_2}{|\gamma_2|}\sqrt{\frac{A_2 \mp B}{2}}z_2\right)=\delta_2q^{\pm}(z)\end{align*} for some $\delta_2 \in \mathbb{T}$ as desired. The second equality above follows from the observation that $A_2^2-B^2=4|\gamma_2|^2.$

Thus in either case, $\mathcal{S}_2 \ominus z_2\mathcal{S}_2$  must have an orthonormal basis of the form $\beta^+_2=\{\delta_2 \frac{q^+}{p}\}$ or $\beta_2^-=\{\delta_2\frac{q^-}{p}\}$, so $\mathcal{S}_2$ equals  $\mathcal{S}_2^+:=\mathcal{H}\left(\frac{K_2^+}{1-z_2\overline{w_2}}\right)$ or $\mathcal{S}_2^{-}:=\mathcal{H}\left(\frac{K_2^-}{1-z_2\overline{w_2}}\right)$ for the kernel functions \begin{align*} K^{+}_2(z,w)=\frac{q^+(z)\overline{q^+(w)}}{p(z)\overline{p(w)}} \quad \text{and} \quad K^{-}_2(z,w)=\frac{q^-(z)\overline{q^-(w)}}{p(z)\overline{p(w)}}.\end{align*}  Similarly $\mathcal{S}_1$ must be $\mathcal{S}_1^+:=\mathcal{H}\left(\frac{K_1^+}{1-z_1\overline{w_1}}\right)$ or $S_1^-:=\mathcal{H}\left(\frac{K_1^-}{1-z_1\overline{w_1}}\right)$where  \begin{align*} K^{+}_1(z,w)=\frac{r^+(z)\overline{r^+(w)}}{p(z)\overline{p(w)}} \quad \text{and} \quad K^{-}_1(z,w)=\frac{r^-(z)\overline{r^-(w)}}{p(z)\overline{p(w)}}.\end{align*}  Note that the quantity $B$ appears in the formulas for both $q^{\pm}$ and $r^{\pm}$ since $A_2^2-4|\gamma_2|^2=A_1^2-4|\gamma_1|^2$, a fact that can be verified by direct computation.

 By the structure of the kernels, it is guaranteed that  $\mathcal{S}_2^+$ and $\mathcal{S}_2^-$ are both $S_{z_2}$-invariant and  $\mathcal{S}_1^+$ and $\mathcal{S}_1^-$  are both $S_{z_1}$-invariant. Although at least one pair of choices for $\mathcal{S}_1$ and $\mathcal{S}_2$ from these possibilities must produce an orthogonal decomposition of $\mathcal{K}_{\theta}$, we do not yet know whether all pairs of choices do.

If $B=0$, $q^+=q^-$ and $r^+=r^-$, so  $\mathcal{S}_i^+=\mathcal{S}_i^-$ for $i=1, 2$, so there is actually only one choice for $\mathcal{S}_1$ and for $\mathcal{S}_2$, and these choices must must yield the desired decomposition.  We claim that $B=0$  occurs if and only if $p$ has a zero in $\mathbb{T}^2$.  To see this, let $\tau_1=\frac{-2\overline{\gamma_1}}{A_1}$ and $\tau_2=\frac{-2\overline{\gamma_2}}{A_2}$, which are well-defined since our earlier computations imply that $A_1$ and $A_2$ are nonzero.   Note that $(\tau_1, \tau_2) \in \mathbb{T}^2$ if and only if $B=0$.  The proof  \cite[Lemma 4.1]{bg17} showed that if $p$ has a zero in $\mathbb{T}^2$, then this zero must be $(\tau_1, \tau_2)$, and one can compute directly that $p(\tau_1, \tau_2)=bB/(A_1A_2)$, so p has a zero at $(\tau_1, \tau_2)$ whenever $(\tau_1, \tau_2) \in \mathbb{T}^2$. 

Now suppose $B \neq 0$.  Then $q^+ \neq q^-$ and $r^+ \neq r^-$, so we have two distinct possibilities for each $\mathcal{S}_i$. Straightforward but tedious computations show that the kernel pairs $(K_1^+, K_2^-)$  and $(K_1^-, K_2^+)$ satisfy \eqref{Agler} and are thus Agler kernels of $\theta$, but $(K_1^+, K_2^+)$  and $(K_1^-, K_2^-)$ do not satisfy \eqref{Agler} since the $B$ terms that appear in the right hand side of the equation  in these cases do not cancel.
By the proof of \cite[Theorem 2.8]{b12}, since $\theta$ does not have a unique Agler decomposition, $\mathcal{S}_1^{\max} \neq \mathcal{S}_1^{\min}$ and $\mathcal{S}_2^{\max} \neq \mathcal{S}_2^{\min}$, so $\mathcal{K}_{\theta}$ has at least two orthogonal decompositions into closed $S_{z_1}$- and $S_{z_2}$-invariant, respectively,  subspaces.  Since each such decomposition must produce a different pair of Agler kernels, we conclude that both $\mathcal{K}_{\theta} = \mathcal{S}_1^+ \oplus \mathcal{S}_2^-$ and $\mathcal{K}_{\theta} = \mathcal{S}_1^- \oplus \mathcal{S}_2^+$ are orthogonal decompositions of this type, and these are the only such decompositions of $\mathcal{K}_\theta$. 

The conclusions about the bases for $\mathcal{S}_1 \ominus z_1\mathcal{S}_1$ and $\mathcal{S}_2 \ominus z_2\mathcal{S}_2$ follow from the preceding discussion.  The formula for $M^1_{\theta}(z_2)$ is obtained from Proposition \ref{prop:n1}  by using \eqref{eqn:shiftformula} to compute $S_{z_1}^*p$.   Applying \eqref{eqn:shiftformula} to $\theta$, we obtain that  $$(S_{z_1}^*\theta)(z)=\lambda\left(\frac{\overline{\gamma_2}+A_2z_2+\gamma_2z_2^2}{p(z)p(0, z_2)}\right),$$ and a routine computation shows that $\overline{\zeta_2}q^+(z)q^-(z)=\overline{\gamma_2}+A_2z_2+\gamma_2z_2^2$.The formulas for $M^2_{\theta}(z_1)$ and $S^*_{z_2}\theta$ are obtained by using analogous versions of Proposition \ref{prop:n1} and  \eqref{eqn:shiftformula} for degree $(m,1)$ RIF and  $S_{z_2}^*$ respectively.    \end{proof} 

We now apply the four preceding results to explore several classes of RIFs built as products of degree $(1,0)$ and/or $(1,1)$ RIFs that will be used as examples in future sections. 
\begin{example} \label{example:1_0prod} Suppose $\Theta$ is a degree $(m,0)$ RIF with $m \geq 1$.  Then $\Theta$  can be written in the form $\Theta = \prod_{t=1}^m \theta_t$ where \begin{align*}\theta_t(z)=\frac{z_1-a_t}{1-\overline{a_t}z_1} \, \,  \text{for} \, \, 1 \leq t \leq m-1 \quad \text{and} \quad  \theta_m(z)=\lambda\frac{z_1-a_m}{1-\overline{a_m}z_1}, \end{align*} for some $\lambda \in \mathbb{T}$ and $a_1, \ldots, a_m \in \mathbb{D}$. 
 We  incorporate $\lambda$ into the final factor to simplify computations.  By Lemma \ref{lemma:1_0},  each $\mathcal{K}_{\theta_t}$ decomposes as $\mathcal{K}_{\theta_t}=\mathcal{S}_{1}^{\theta_t} \oplus \mathcal{S}_2^{\theta_t}$  where the orthonormal basis $\beta_2^{\theta_t}=\{f_1^t\}$ for $\mathcal{S}_2^{\theta_t} \ominus z_2 \mathcal{S}_{2}^{\theta_t}$ is defined by setting $f_1^t(z)=\frac{\sqrt{1-|a_t|^2}}{1-\overline{a_t}z_1}$.  Moreover, for $1 \leq t \leq m $, $M^1_{\theta_t, \beta^{\theta_t}_2}$ is the constant function $M_{\theta_t, \beta^{\theta_t}_2}^1(z)=a_t$, and, for $1 \leq t \leq m-1$, $S_{z_1}^*\theta_t=g_1^tf_1^t$ for the constant function $g_1^t(z)=\sqrt{1-|a_t|^2}$. 
 
 For $i=1, 2$, let $\mathcal{S}_i=\bigoplus_{t=1}^m \left(\left(\prod_{k=1}^{t-1} \theta_k\right)\!\mathcal{S}^{\theta_t}_i\right)$.    By Proposition \ref{prop:productdecomp}, $\mathcal{K}_{\Theta}=\mathcal{S}_1 \oplus \mathcal{S}_2$ is an orthogonal decomposition of $\mathcal{K}_{\Theta}$ into closed subspaces that are respectively $S_{z_1}$- and $S_{z_2}$-invariant and $\beta_2=\left\{f_1^1, \theta_1f_1^2, \ldots \left(\prod_{k=1}^{m-1}\theta_k\right)f_1^m\right\}$  is an orthonormal basis for $\mathcal{S}_2 \ominus z_2\mathcal{S}_2$. Note that if the functions in $\beta_2$ are viewed as functions of one variable instead of two, then $\beta_2$ is the standard ordering of the Takenaka-Malmquist-Walsh basis corresponding to $\Theta$ as a single-variable function. Applying the formula from Theorem \ref{thm:productM}, we obtain that  $M^1_{\Theta, \beta_2}$ is the $m \times m$ matrix 
$M^1_{\Theta, \beta_2}=[H_{st}]$ with scalar-valued entries \begin{align*}H_{st} = \left\{\begin{array}{ll}  0 & \text{if } s< t \\  a_s, &  \text{if } s=t\\\left(\prod_{\ell=t+1}^{s-1}-\overline a_{\ell}\right)\sqrt{1-|a_s|^2}\sqrt{1-|a_t|^2}, & \text{if } s>t\end{array} \right. . \end{align*}  Note that $M^1_{\Theta, \beta_2}$ is the transpose of the matrix $M_{\Theta}$ given by \eqref{eqn:SB} when $\Theta$ is viewed as a function of a single variable,  and thus it is the transpose of the $M^1_{\Theta}$ matrix obtained for $\Theta$ in \cite[Example 3.3]{bg17}.  The appearance of the transpose here is the result of utilizing a  different version of a Takenaka-Malmquist-Walsh basis than the one utilized in these other computations.
 \end{example}

\begin{example} \label{example:1_1prod} Suppose $m \in \mathbb{N}$ and $\Theta=\prod_{t=1}^m \theta_t$ where each $\theta_t=\lambda_t\frac{\tilde{p}_t}{p_t}$ is a degree (1,1) RIF and, for notational simplicity, $\lambda_t=1$ for all $t \neq m$. For each $t$ and $i=1,2$, let $\gamma_{i, t}$, $\zeta_{i,t}$, $A_{i,t}$, $B_t$, $q_{t}^{\pm}$, and $r_t^{\pm}$ be the values of $\gamma_i$, $\zeta_i$, $A_i$, $B$, $q^{\pm}$, and $r^{\pm}$, respectively,  from Lemma \ref{lemma:1_1} for the function $\theta_t$.  For each $t$, let $f_1^{t,2}=\frac{q^+_t}{p_t}$ and $f_1^{t,1}=\frac{r^-_t}{p_t}$.  Then by Lemma \ref{lemma:1_1}, we can choose to decompose $\mathcal{K}_{\theta_t}$ as  $\mathcal{K}_{\theta_t}=\mathcal{S}_{1}^{\theta_t} \oplus \mathcal{S}_2^{\theta_t}$  using  $\beta_i^{\theta_t}=\{f_1^{t,i}\}$ as the orthonormal basis for $\mathcal{S}_i^{\theta_t} \ominus z_i \mathcal{S}_{i}^{\theta_t}$ for $i=1,2$. Moreover, for all $1 \leq t \leq m-1$, $S_{z_1}^*{\theta_t}=g_1^{t}f_1^{t,2}$ for  $g_1^{t}(z_2)=\overline{\zeta_{2,t}}\frac{q^-_t(z_2)}{p_t(0,z_2)}$ and $S_{z_2}^*{\theta_t}=g_1^{t,1}\!f_{1}^{t,1}$ for   $g_1^{t,1}(z_1)=\overline{\zeta_{1,t}}\frac{r^+_t(z_1)}{p_t(z_1,0)}$.

As in the previous example, by Proposition \ref{prop:productdecomp} and the following discussion, we can now decompose $\mathcal{K}_{\Theta}$ into respectively $S_{z_1}-$ and $S_{z_2}-$invariant subspaces as $\mathcal{K}_{\Theta}=\mathcal{S}_1 \oplus \mathcal{S}_2$ where, for $i=1, 2$,  $\beta_i=\left\{f_1^{1, i}, \theta_1f_1^{2,i}, \ldots \left(\prod_{k=1}^{m-1}\theta_k\right)f_1^{m,i}\right\}$ is an orthonormal basis for $\mathcal{S}_i \ominus z_i\mathcal{S}_i$.  By Theorem \ref{thm:productM} and the following discussion, $M_{\Theta, \beta_2}^1=[H_{st}]$ and $M_{\Theta, \beta_1}^2=[\hat{H}_{st}]$  are both $m \times m$ matrices where the  $1 \times 1$ blocks $H_{st}$ and $\hat{H}_{s,t}$ are defined by   \begin{align*} H_{st} = \left\{\begin{array}{ll}  0, & \text{if } s< t \\  M_{\theta_s}^1, &  \text{if } s=t\\
\zeta_{2,t}\left(\prod_{\ell=t+1}^{s-1}\overline{\theta_{\ell}(0, z_2)}\right)\overline{\frac{q^+_s(z_2)q^-_t(z_2)}{p_s(0,z_2)p_t(0,z_2)}}& \text{if } s>t\end{array} \right.\end{align*} and 
 \begin{align*} \hat{H}_{st} = \left\{\begin{array}{ll}  0, & \text{if } s< t \\  M_{\theta_s}^2, &  \text{if } s=t\\
\zeta_{1,t}\left(\prod_{\ell=t+1}^{s-1}\overline{\theta_{\ell}(z_1, 0)}\right)\overline{\frac{r^-_s(z_1)r^+_t(z_1)}{p_s(z_1, 0)p_t(z_1, 0)}}& \text{if } s>t\end{array} \right.,\end{align*}  respectively and each  $M^i_{\theta_s}$ is defined in terms of the coefficients of $p_s$ as in Lemma \ref{lemma:1_1}. 

We now consider a special case in which this example simplifies significantly.  Suppose $B(z)=\lambda \prod_{t=1}^m \frac{z-a_t}{1-\overline{a_t}z}$ is a single-variable finite Blaschke product and $\Theta(z)=B(z_1z_2)$.  Then letting $\theta_t=\frac{z_1z_2-a_t}{1-\overline{a_t}z_1z_2}$ for $1 \leq t \leq m-1$ and $\theta_m=\lambda \frac{z_1z_2-a_m}{1-\overline{a_m}z_1z_2}$, we have that, for all $1 \leq t \leq m$,  $p_t(z)=1+0z_1+0z_2+(-\overline{a_t})z_1z_2$ so, by Lemma \ref{lemma:1_1}, $M^1_{\theta_t}=a_t\overline{z_2}$,  $M^2_{\theta_t}=a_t\overline{z_1}$, $\gamma_{1,t}=\gamma_{2,t}=0$, $\zeta_{1,t}=\zeta_{2,t}=1$, and $A_{1,t}=A_{2,t}=B_t=1-|a_t|^2$.  Hence $q_t^+(z_2)=r_t^+(z_1)=\sqrt{1-|a_t|^2}$, $q^-_t(z_2)\sqrt{1-|a_t|^2}z_2$, and $r^-_t(z_1)=\sqrt{1-|a_t|^2}z_1.$  Note that for all $t$ with $1 \le t \le m-1$, $\theta_t(0, z_2)=\theta_t(z_1,0)=-a_t$ and $p_t(0, z_2)=p_t(z_1,0)=1$.  Thus, the formula for  $H_{st}$ simplifies to   \begin{align*}H_{st} = \left\{\begin{array}{ll}  0 & \text{if } s< t \\  a_s\overline{z_2}, &  \text{if } s=t\\\left(\prod_{\ell=t+1}^{s-1}-\overline a_{\ell}\right)\sqrt{1-|a_s|^2}\sqrt{1-|a_t|^2}\overline{z_2}, & \text{if } s>t\end{array} \right., \end{align*} and the formula for $\hat{H}_{st}$ is the above with $z_2$ replaced by $z_1$.  Thus, $M_{\Theta, \beta_2}^1=\overline{z_2}M_B^T$  and $M_{\Theta, \beta_1}^2=\overline{z_1}M_B^T$, where the scalar matrix $M_B$ is defined by \eqref{eqn:SB} and $T$ denotes the transpose.
\end{example} 

\begin{example} \label{example:z_1product} Suppose $k \in \mathbb{N}$ and $\theta$ is any degree $(m,n)$ RIF with $m \geq 1$.  Define $\Theta(z)=z_1^k\theta(z)$.  
By Example \ref{example:1_0prod}, we can decompose $\mathcal{K}_{z_1^k}$ as $\mathcal{K}_{z_1^k}=\mathcal{S}_1^{z_1^k} \oplus \mathcal{S}_2^{z_1^k}$ using the orthonormal basis $\hat{\beta_2}=\{1, z_1, z_1^2, \ldots, z_1^{k-1}\}$ for $\mathcal{S}_{2}^{z_1^k} \ominus z_2\mathcal{S}_2^{z_1^k}$, and $M^1_{z_1^k, \hat{\beta_2}}$ \ is the $k \times k$ matrix that has 1's on the first subdiagonal and 0's elsewhere.  Setting $f_i^1(z)=z_1^{i-1}$ for $1 \leq i \leq k$, the functions $g_1^1, \ldots, g_k^1 \in H^2_2(\mathbb{D})$ such that $S_{z_1}^*z_1^k=\sum_{j=1}^{k}g_j^1f_j^1$ are the constant functions defined by $g_j^1(z_2)=0$ if $1 \leq j \leq k-1$ and $g_k^1(z_2)=1$. 

Fix any decomposition  $\mathcal{K}_{\theta}=\mathcal{S}_1^{\theta} \oplus \mathcal{S}_2^{\theta}$ of $\mathcal{K}_{\theta}$ into respectively $S_{z_1}$- and $S_{z_2}$-invariant subspaces and any orthonormal basis $\beta_2^{\theta}=\{f_1^2, \ldots, f_m^2\}$ of $\mathcal{S}_2^{\theta} \ominus z_2\mathcal{S}_2^{\theta}$.  Let $g_1^2, \ldots, g_m^2 \in H^2_2(\mathbb{D})$ be the functions that satisfy $S_{z_1}^*\theta=\sum_{j=1}^{m}g_j^2f_j^2$.

We can now apply Proposition \ref{prop:productdecomp} and Theorem \ref{thm:productM} to $\Theta$ considering $z_1^k$ as the first function in the product and $\theta$ as the second, which yields a decomposition $\mathcal{K}_{\Theta}=\mathcal{S}_1 \oplus \mathcal{S}_2$ where $\beta_2=\{1, z_1, \ldots, z_1^{k-1}, z_1^kf_1^2, z_1^kf_2^2, \ldots, z_1^kf_m^2\}$ is an orthonormal basis for $\mathcal{S}_2 \ominus z_2\mathcal{S}_2$ and $M^1_{\Theta, \beta_2}$  is the block matrix \begin{align*} M_{\Theta, \beta_2}^1=\left[\begin{array}{cc} M_{z_1^k, \hat{\beta_2}}^1 & 0_{k \times m} \\ C & M_{\theta, \beta_2^{\theta}}^1\end{array}\right]\end{align*}  where $C=[c_{ij}]$ is the $m \times k$ matrix-valued function with entries  \begin{align*}c_{ij}=\overline{f_i^2(0, z_2)g_j^1(z_2)}= \left\{\begin{array}{ll} \overline{f_i^2(0, z_2)}, & \text{if} \, j=k \\ 0, & \text{otherwise} \end{array}\right.\end{align*}

The  entries in $C$ are particularly simple  if $\theta(z)=B(z_1z_2)$ for a finite Blaschke product  $B(z)=\lambda \prod_{t=1}^m \frac{z-a_t}{1-\overline{a_t}z}$ and we define $\theta_1, \ldots, \theta_m$ as in Example  \ref{example:1_1prod} and choose the decomposition of $\mathcal{K}_{\theta}$ and orthonormal basis $\beta_2^2=\{f_1^2, \ldots, f_m^2\}$ as in that example.  Recall that $f_i^2= \left(\prod_{\ell=1}^{i-1}\theta_\ell\right)\tilde{f}_1^i$ for $\tilde{f}_1^i=\frac{\sqrt{1-|a_i|^2}}{1-\overline{a_i}z_1z_2}$ for all $1 \leq i \leq m$.  
Hence, in this case, \begin{align*} c_{ij}=\left\{\begin{array}{ll} \left(\prod_{\ell=1}^{i-1} -\overline{a_{\ell}}\right) \sqrt{1-|a_i|^2}, & \text{if} \, j = k\\ 0, & \text{otherwise}\end{array}\right..\end{align*} 
Returning to the general case, note that we could also consider $\theta$ as the first function and $z_1^k$ as the second.  This yields a decomposition  $\mathcal{K}_{\theta}=\tilde{\mathcal{S}}_1 \oplus \tilde{\mathcal{S}}_2$ with orthonormal basis $\widetilde{\beta_2}=\{f_1^2, f_2^2, \ldots, f_m^2, \theta, \theta z_1, \ldots, \theta z_1^{k-1}\}$ for $\tilde{\mathcal{S}}_2 \ominus z_2\tilde{\mathcal{S}}_2$.  Then  \begin{align*} M_{\Theta, \widetilde{\beta_2}}^1=\left[\begin{array}{cc} M_{\theta, \beta^{\theta}_2}^1 & 0_{m \times k} \\ D & M_{z_1^k, \hat{\beta_2}}^1 \end{array}\right]\end{align*} where $D=[d_{ij}]$ is the $k \times m$ matrix-valued function with entries 
\begin{align*} d_{ij}=\overline{f_i^1(0, z_2)g_j^2(z_2)}=\left\{\begin{array}{ll} \overline{g_j^2(z_2)}, & \text{if} \, i=1 \\ 0, & \text{otherwise}\end{array}\right..\end{align*}
The matrix $D$ also simplifies if $\theta(z)=B(z_1z_2)$ for a finite Blaschke product $B$, but we leave those computations to the interested reader. \end{example} 

Notice that in the previous example, we  found two  $M^1_{\Theta}$ matrices for  $\Theta$ that will, in at least most cases, be different.  This raises some key questions:  \emph{If we have two decompositions with corresponding $\beta_2$ and $\widetilde{\beta_2}$, what relationships, if any, exist between $M^1_{\theta, \beta_2}$ and $M^1_{\theta, \widetilde{\beta_2}}$?  How different can they be?} We'll answer these questions as part of our explorations in the next section.


\section{Uniqueness of $M^1_\theta$} \label{sec:M}

Now that we have a method of generating $M_\theta^1$ examples, we will generalize several of our results for degree $(1,n)$ RIFs from Section \ref{sec:1n}. First, recall that Proposition \ref{prop:n1} implies that if $\deg \theta = (1,n)$,
then $M^1_\theta$ is independent of the chosen decomposition of $\mathcal{K}_\theta$ and associated orthonormal basis. Meanwhile, Proposition \ref{prop:n1B} characterizes when two degree $(1,n)$ RIFs possess the same $M^1_\theta$ function.

In this section, we first generalize Proposition \ref{prop:n1B} to all RIFs of two variables and then use that result to answer Question $1$ about the uniqueness of $M_\theta^1$ and $M_\theta^2$ in Corollary  \ref{cor:Munique}. We then apply this answer to the setting of one RIF $\theta$ with different decompositions of $\mathcal{K}_\theta$ as described at the end of Section \ref{sec:product}. 
Example \ref{ex:Mtheta_2} uses the constructions from Section \ref{sec:product} to illustrate the subtleties of this situation.

Here is some notation for this section. Let  $\theta = \lambda_1\frac{\tilde{p}}{p}$, $\phi = \lambda_2\frac{\tilde{q}}{q}$ be RIFs on $\mathbb{D}^2$ with $\deg_1 (\theta), \deg_1(\phi) \ge 1$ and with associated model space decompositions $\mathcal{K}_\theta = \mathcal{S}_1 \oplus \mathcal{S}_2$ and  $\mathcal{K}_\phi = \widetilde{\mathcal{S}}_1 \oplus \widetilde{\mathcal{S}}_2$. Let $S_\theta^1$ and $S_\phi^1$ denote the associated compressions of the shift in $z_1$, i.e. $S_{z_1}$, defined by
\[ S_\theta^1 = P_{\mathcal{S}_2} S_{z_1} |_{\mathcal{S}_2} \ \ \text{ and }  \ \  S_\phi^1 = P_{\widetilde{\mathcal{S}}_2} S_{z_1} |_{\widetilde{\mathcal{S}}_2}.\]
Then $m:=\deg_1( \theta) = \dim ( \mathcal{S}_2 \ominus z_2 \mathcal{S}_2)$ and $\tilde{m}:=\deg_1( \phi) = \dim ( \widetilde{\mathcal{S}}_2 \ominus z_2 \widetilde{\mathcal{S}}_2)$. 
Let $\beta_2=\{\frac{q_1}{p}, \dots, \frac{q_m}{p}\}$ be an orthonormal basis of $\mathcal{S}_2 \ominus z_2 \mathcal{S}_2$ of form \eqref{eqn:S2basis} and 
$\widetilde{\beta}_2=\{\frac{t_1}{q}, \dots, \frac{t_{\tilde{m}}}{q}\}$ be an orthonormal basis of $\widetilde{\mathcal{S}}_2 \ominus z_2 \widetilde{\mathcal{S}}_2$ also of form \eqref{eqn:S2basis}.
Let $M^1_\theta$ and $M^1_\phi$ denote the associated matrix-valued functions from Theorem \ref{thm:Mconstructive} defined using these bases.  We can fully determine when these two functions are point-wise unitarily equivalent on $\mathbb{T}$.  

\begin{theorem} \label{thm:uniqueness} Given the above setup, there is an $m \times m$ unitary matrix-valued function $U(\cdot)$ on $\mathbb{T}$ such that 
\begin{equation} \label{eqn:MthetaU} M^1_\theta(\tau) = U(\tau) M^1_\phi(\tau) U(\tau)^*\end{equation}
for all $\tau \in \mathbb{T}$ if and only if there are finite Blaschke products $B_1, B_2$ such that 
\begin{equation} \label{eqn:B1B2} B_1(z_2) \phi(z_1, z_2) = B_2(z_2) \theta(z_1, z_2).\end{equation}
\end{theorem}

\begin{proof} ($\Rightarrow$) Assume that such a matrix-valued function $U$ exists. Fix $\tau \in \mathbb{T} \setminus  (E_\phi \cup E_\theta).$ 
Then the dimensions of $M^1_\theta(\tau)$ equal the dimensions of $M^1_\phi(\tau)$, so the degrees of the slice functions  $\theta_\tau$ and $\phi_\tau$ are equal and $m = \tilde{m}$.
$M^1_\theta(\tau)$ and $M^1_\phi(\tau)$ certainly have the same eigenvalues, and by the discussion in Section \ref{sec:2var}, these eigenvalues are the zeros of  $\theta_{\tau}$ and $\phi_{\tau}$, respectively.  Thus, $\theta_\tau$ and $\phi_\tau$ are finite Blaschke products with exactly the same set of zeros (including multiplicity). This means that there is a constant $\lambda(\tau) \in \mathbb{T}$ such that for all $z_1 \in \mathbb{C}$, 
\begin{equation} \label{eqn:1} \theta(z_1, \tau) = \theta_\tau(z_1) =\lambda(\tau) \phi_\tau(z_1) =\lambda(\tau) \phi(z_1, \tau).\end{equation}
In particular, for each  $\tau \in  \mathbb{T} \setminus  (E_\phi \cup E_\theta)$, there is a  $\lambda(\tau)\in \mathbb{T}$ such that \eqref{eqn:1} holds.

To prove the claim, define the rational function $F$ by 
\[ F(z_1, z_2) = \frac{\theta(z_1, z_2)}{\phi(z_1, z_2)},\]
where we cancel any common factors
 and note that we need to show that $F$ is a ratio of two finite Blaschke products.
Choose $(w_1, \gamma) \in \mathbb{D} \times \mathbb{T}$ such that $\phi(w_1, \gamma) \ne 0$ and $\gamma \not \in E_\phi \cup E_\theta.$ Then $F$ is analytic near $(w_1, \tau)$, so we can write 
\[ F(z_1, z_2) = F_1(z_2) + (z_1-w_1) F_2(z_1, z_2),\]
for functions $F_1, F_2$ analytic  in a neighborhood $V \subseteq \mathbb{C}^2$ of $(w_1, \gamma)$. By shrinking $V$ if necessary, we can assume $V = V_1 \times V_2$ where $V_1, V_2$ are open disks centered around $w_1, \gamma$, $V_2 \cap (E_\phi\cup E_\theta) = \emptyset$, and $\phi$ does not vanish on $V$. 
Fix any $\tau \in V_2 \cap \mathbb{T}$ and $z_1 \in V_1$. Then \eqref{eqn:1} shows that 
\[ F_1(\tau) + (z_1-w_1) F_2(z_1, \tau) = \lambda(\tau) =  F_1(\tau) + (w_1-w_1) F_2(w_1, \tau),\]
so $F_2(z_1, \tau) =0$. This implies that for each $z_1 \in V_1$, the one-variable holomorphic function $G_{z_1}(z_2): = F_2(z_1, z_2)$ vanishes on the curve $V_2 \cap \mathbb{T}$. This implies $G_{z_1} \equiv 0$, so $F_2 \equiv 0$. Thus, in a neighborhood of $(w_1, \gamma)$, $F(z_1, z_2) = F_1(z_2)$ depends only on $z_2$. Then $\frac{\partial F}{\partial z_1} \equiv 0$ on that neighborhood and hence on the entire domain of $F$, so $F$ is a rational function in $z_2$ only. 

By \eqref{eqn:1}, we also know that $|F(\tau)| =1$ for almost every $\tau \in \mathbb{T}$. Since we already cancelled any common factors from the numerator and denominator of $F$, this implies that $F$ must be continuous on a neighborhood of $\mathbb{T}$.  Then \cite[p. 50, Corollary 3.5.4]{gmr18} implies that $F = \frac{B_1}{B_2}$, for some pair of finite Blaschke products $B_1$ and $B_2$. It follows immediately that on $\mathbb{D}^2$,
\[  \frac{\theta(z_1, z_2)}{\phi(z_1, z_2)} = \frac{B_1(z_2)}{B_2(z_2)},\]
which establishes \eqref{eqn:B1B2}.

($\Leftarrow$) For the reverse direction, assume \eqref{eqn:B1B2} holds and fix $\tau \in \mathbb{T}\setminus(E_{\phi} \cup E_{\theta})$.  Notice that $\theta_{\tau}(z_1)=\lambda \phi_{\tau}(z_1)$ for $\lambda = B_1(\tau)/B_2(\tau) \in \mathbb{T}$, which implies that $\mathcal{K}_{\theta_\tau}=\mathcal{K}_{\phi_\tau}$ and hence $S_{\theta_\tau}=S_{\phi_\tau}$.  Thus, $M^1_{\theta}(\tau)$ and $M^1_{\phi}(\tau)$ are both matrix representations of the same operator but potentially with respect to different orthonormal bases, so there must exist a unitary change of basis matrix $U(\tau)$ such that \eqref{eqn:MthetaU} holds.  

To show that the equality actually holds for all $\tau \in \mathbb{T}$ and not just $\mathbb{T}\setminus(E_{\phi} \cup E_{\theta})$, we construct an explicit formula for $U(\tau)=[u_{ij}(\tau)]$. By change of basis, we know that, for each $\tau \in \mathbb{T}\setminus(E_{\phi} \cup E_{\theta})$ and $1 \leq j \leq m$, $[u_{1j}(\tau),  \ldots, u_{mj}(\tau)]^{t}$ is the unique solution to
\begin{align}\label{eqn:unitary1} \sum_{i=1}^m u_{ij}(\tau)\frac{q_i}{p}(\cdot, \tau)= \frac{t_j}{q}(\cdot, \tau). \end{align}  As noted in the proof of Proposition \ref{prop:n1B}, \eqref{eqn:B1B2} implies that $r_1(z_2)q(z)=cr_2(z_2)p(z)$ for some nonzero constant $c$, where $r_1$ and $r_2$ are the denominators of the Blaschke products $B_1$ and $B_2$, respectively, and thus do not vanish on $\overline{\mathbb{D}}$.  Thus, by multiplying both sides of \eqref{eqn:unitary1} by $r_2(\tau)p(\cdot, \tau)$, we see that $U(\tau)$ must satisfy 
\begin{align}\label{eqn:unitary2} \sum_{i=1}^m u_{ij}(\tau)r_2(\tau)q_i(\cdot, \tau) = \frac{1}{c}r_1(\tau)t_{j}(\cdot, \tau)\end{align} for all $1 \leq j \leq m$ and $\tau \in \mathbb{T}\setminus(E_{\phi} \cup E_{\theta})$.  Note that, for all $1 \leq i \leq m$, $\hat{q_i}(z):=r_2(z_2)q_i(z)$ and $\hat{t_i}(z):=\tfrac{1}{c}r_1(z_2)t_i(z)$ are polynomials in $z_1$ and $z_2$ whose degree in $z_1$ is at most $m-1$.  Thus, there exist polynomials $q_{ik}(z_2), t_{ik}(z_2)$ for $1 \leq i \leq m, 0 \leq k \leq m-1$ such that 
\begin{align*}\hat{q_i}(z)=\sum_{k=0}^{m-1} q_{ik}(z_2)z_1^k \qquad \hat{t_i}(z)=\sum_{k=0}^{m-1}t_{ik}(z_2)z_1^k.\end{align*}
Let $A(z_2)=[q_{j(i-1)}(z_2)]$ and $B(z_2)=[t_{j(i-1)}(z_2)]$.  Then satisfying \eqref{eqn:unitary2} for all $1 \leq j \leq m$ is equivalent to satisfying $A(\tau)U(\tau)=B(\tau).$ For all  $\tau \in \mathbb{T}\setminus(E_{\phi} \cup E_{\theta})$, $\det(A(\tau)) \neq 0$ since otherwise one of $r_2(\tau)q_1(\cdot, \tau), \ldots, r_2(\tau)q_m(\cdot, \tau)$ would be a linear combination of the others contradicting the assumption that $\frac{q_1}{p}(\cdot, \tau), \ldots \frac{q_m}{p}(\cdot, \tau)$ is a basis for $\mathcal{K}_{\theta_{\tau}}$.  
Thus, by Cramer's Rule,  for $\tau \in  \mathbb{T}\setminus(E_{\phi} \cup E_{\theta})$,$$u_{ij}(\tau)=\frac{\det(A_{B}^{i,j}(\tau))}{\det(A(\tau))},$$ where $A_{B}^{i,j}(\tau)$ is formed from $A(\tau)$ by replacing the $i$th column of $A(\tau)$ by the $j$th column of $B(\tau)$.   Note that $U(\tau)$ is the boundary function of $U(z_2)=[\det(A_{B}^{i,j}(z_2))/\det(A(z_2))]$, which  is a matrix-valued function whose entries are all rational functions in $z_2$ having a common set of singularities.  Since $U(z_2)$ is unitary for all but at most finitely many points in $\mathbb{T}$, all of these singularities that occur in $\mathbb{T}$ must be removable singularities.  Removing the singularities, we obtain  that $U$ must now be unitary-valued on $\mathbb{T}$ and \eqref{eqn:1} must hold for all $\tau \in \mathbb{T}$ since $M^1_{\theta}$ and $M^1_{\phi}$ are continuous on $\overline{\mathbb{D}}$. 
\end{proof} 

Versions of Theorem \ref{thm:Mconstructive} and Theorem \ref{thm:uniqueness} hold if we switch the roles of $z_1$ and $z_2$, require $\deg_2 \theta, \deg_2 \phi \ge 1$, and define 
\[ S_\theta^2 = P_{\mathcal{S}_1} S_{z_2} |_{\mathcal{S}_1} \ \ \text{ and }  \ \  S_\phi^2 = P_{\widetilde{\mathcal{S}}_1} S_{z_2} |_{\widetilde{\mathcal{S}}_1}.\]
Let $M_\theta^1,$ $M_\theta^2$, $M_\phi^1$, $M_\phi^2$ be the matrix-valued rational functions for $S_\theta^1$, $S_\theta^2$, $S_\phi^1$, $S_\phi^2$ guaranteed by Theorem \ref{thm:Mconstructive}. Using this notation, we have the following corollary, which fully answers Question 1 by characterizing when $\theta = \lambda \phi$ using only the symbols associated to the two-variable compressions of shifts.

\begin{corollary} \label{cor:Munique}There are unitary matrix-valued functions $U_1, U_2$ on $\mathbb{T}$ such that 
\[ 
M^1_\theta(\tau) = U_1(\tau) M^1_\phi(\tau) U_1(\tau)^* \quad \text{ and } \quad M^2_\theta(\tau) = U_2(\tau) M^2_\phi(\tau) U_2(\tau)^* \quad  \] for all  $\tau \in \mathbb{T}$
  if and only  $\theta=\lambda \phi,$ for some $\lambda \in \mathbb{T}$.
\end{corollary} 

\begin{proof} The backward direction is immediate from Theorem \ref{thm:uniqueness}. To prove the forward direction, just note that the forward direction of the proof of Theorem \ref{thm:uniqueness} implies that $\frac{\theta}{\phi}$ is a rational function of only $z_2$ and a symmetric argument says  $\frac{\theta}{\phi}$ is a rational function of only $z_1$. Thus, $\frac{\theta}{\phi} \equiv \lambda$ for some $\lambda \in \mathbb{C}$. Since $|\theta(\tau)| = |\phi(\tau)| =1$ for almost every $\tau \in \mathbb{T}^2$, we must have $\lambda \in \mathbb{T}$. 
\end{proof}

Theorem \ref{thm:uniqueness} also shows that for a single $\theta$,  the choice of decomposition of $\mathcal{K}_\theta$ has a limited effect on the resulting compressed shift $S_\theta^1$:

\begin{corollary} \label{cor:2decomps} Assume that  $\mathcal{K}_\theta = \mathcal{S}_1 \oplus \mathcal{S}_2= \widetilde{\mathcal{S}}_1 \oplus \widetilde{\mathcal{S}}_2$ gives two orthogonal decompositions of $\mathcal{K}_\theta$ and let 
\[ S_\theta^1 = P_{\mathcal{S}_2} S_{z_1} |_{\mathcal{S}_2} \ \ \text{ and }  \ \ \tilde{S}_\theta^1 = P_{\tilde{\mathcal{S}}_2} S_{z_1} |_{\tilde{\mathcal{S}}_2}.\]
For any choices of orthonormal bases $\beta_2$ and $\widetilde{\beta_2}$ of $\mathcal{S}_2 \ominus z_2 \mathcal{S}_2$ and $\tilde{\mathcal{S}}_2 \ominus z_2 \tilde{\mathcal{S}}_2$, respectively, the associated matrix-valued functions $M^1_{\theta, \beta_2}$ and $M^1_{\theta, \widetilde{\beta_2}}$ from Theorem \ref{thm:Mconstructive} satisfy
\[  M^1_{\theta, \beta_2}(\tau) = U(\tau)M^1_{\theta, \widetilde{\beta_2}}(\tau) U(\tau)^*,\]
where each $U(\tau)$ is an $m\times m$ unitary matrix and $\tau \in \mathbb{T}.$
\end{corollary}

We illustrate this Corollary \ref{cor:2decomps} with the following example:

\begin{example} \label{ex:Mtheta_2} Let $\theta(z) = \frac{2z_1z_2-z_1-z_2}{2-z_1-z_2}$ and set $\Theta(z) = z_1 \theta(z).$  To put us in the setting of Corollary \ref{cor:2decomps}, we need to find two orthogonal decompositions of $\mathcal{K}_{\Theta}$.  We begin by considering $\theta$.  Applying Lemma \ref{lemma:1_1} to $\theta$, we find that $A_2=4$, $\gamma_2=-2$, $\zeta_2=-1$, $B=0$, and $q^{\pm}(z_2)=\sqrt{2}(1-z_2)$,  so there is a unique decomposition of $\mathcal{K}_{\theta}=\mathcal{S}^{\theta}_1 \oplus \mathcal{S}^{\theta}_2$ into respectively $S_{z_1}$- and $S_{z_2}$-invariant subspaces and, setting   $f_1^2(z):=\frac{\sqrt{2}(1-z_2)}{2-z_1-z_2}$, $\beta_2^{\theta}:=\{f_1^2\}$ is an orthonormal basis for $\mathcal{S}_2^{\theta} \ominus z_2\mathcal{S}_2^{\theta}$. Moreover, $M^1_{\theta}=\frac{1}{2-\overline{z_2}}$ and $S_{z_1}^*\theta=g_1^2f_1^2$ for $g_1^2(z_2)=-\frac{\sqrt{2}(1-z_2)}{2-z_2}$ since $\zeta_2=-1$. 

By Example \ref{example:z_1product}, we have two decompositions of $\mathcal{K}_{\Theta}$ as $\mathcal{K}_{\Theta}=\mathcal{S}_1 \oplus \mathcal{S}_2$ with orthonormal basis $\beta_2:=\{1, z_1f_1^2\}$ for $\mathcal{S}_2 \ominus z_2\mathcal{S}_2$ and as $\mathcal{K}_{\Theta}=\tilde{S}_1 \oplus \tilde{\mathcal{S}}_2$ with orthonormal basis $\widetilde{\beta_2}:=\{f_1^2, \theta\}$.  Thus, by the formulas in that example,  \[M^1_{\theta, \beta_2}=\frac{1}{2-\overline{z_2}}\begin{bmatrix}0 & 0 \\ \sqrt{2}(1-\overline{z_2}) & 1\end{bmatrix}  \quad \text{and} \quad M^1_{\theta, \widetilde{\beta_2}} = \frac{1}{2-\overline{z_2}}\begin{bmatrix} 1 & 0 \\ -\sqrt{2}(1-\overline{z_2}) & 0 \end{bmatrix}. \]
since as noted in the example, $M^1_{z_1}=0$.

We can now apply the algorithm from the proof of Theorem \ref{thm:uniqueness} to find $U(\tau)$.  Since $p=q=2-z_1-z_2$, $$\hat{q_1}=(2-z_2)+(-1)z_1, \quad \hat{q_2}=0+\sqrt{2}(1-z_2)z_1$$ and $$\hat{t_1}=\sqrt{2}(1-z_2)+0z_1, \quad  \hat{t_2}=-z_2+(2z_2-1)z_1.$$  
Thus, $A(z_2)=\begin{bmatrix} 2-z_2 & 0\\ -1 & \sqrt{2}(1-z_2)\end{bmatrix},$ $B(z_2)=\begin{bmatrix}  \sqrt{2}(1-z_2) & -z_2 \\ 0 & 2z_2-1 \end{bmatrix},$  and \begin{align*} U(z_2)&=\frac{1}{\sqrt{2}(2-z_2)(1-z_2)}\begin{bmatrix} 2(1-z_2)^2 &  -\sqrt{2}z_2(1-z_2) &  \\ \sqrt{2}(1-z_2) &-2(1-z_2)^2 \end{bmatrix}\\&=\frac{1}{2-z_2}\begin{bmatrix}  \sqrt{2}(1-z_2) &  -z_2 \\  1 & -\sqrt{2}(1-z_2) \end{bmatrix}\end{align*} after cancelling common factors. 
It is easy to double check that, for all $\tau \in \mathbb{T}$,  $U(\tau)$ is a unitary matrix and $M^1_{\theta, \beta_2}(\tau) = U(\tau) M^1_{\theta, \widetilde{\beta_2}}(\tau) U(\tau)^*$.

Note that since $U$ and $M_{\theta, \beta_2}^{1 \, *}$ are bounded and analytic on $\mathbb{D}$ and $U$ is unitary-valued on $\mathbb{T}$, it follows from basic facts about matrix-valued Toeplitz operators that $T_U$ is an isometry and $T_{U}^*T_{M^1_{\theta, \beta_2}}T_{U}=T_{M^1_{\theta, \widetilde{\beta_2}}}$.  However, $T_U$ is not a unitary operator and $T_UT_{M^1_{\theta, \widetilde{\beta_2}}}T_U^* \neq T_{M^1_{\theta}, \beta_2}$. To see this, let $\vec{r}(z)=\begin{bmatrix}0 \\ 1\end{bmatrix}$.  Then $||\vec{r}||_{H^2_2}=1$ but \[T_U^*\vec{r}=\begin{bmatrix} P_{H^2_2}\left(\frac{1}{2-\overline{z_2}}\right) \\P_{H^2_2}\left(\frac{-\sqrt{2}(1-\overline{z_2})}{2-\overline{z_2}}\right) \end{bmatrix}=\begin{bmatrix}   \frac{1}{2}\\  -\frac{\sqrt{2}}{2}\end{bmatrix},\] which is not norm one. Above, $P_{H^2_2}$ denotes the orthogonal projection from $L^2(\mathbb{T})$ with variable $z_2$ to $H^2_2$.  Moreover, similar computations show that \[T_{U}T_{{M}^1_{\theta, \widetilde{\beta_2}}}T_{U}^*\vec{r}=\frac{1}{4(2-z_2)}\begin{bmatrix}\sqrt{2}\\ 3-2z_2\end{bmatrix}\] while the first component of $T_{M^1_{\theta, \beta_2}}\vec{r}$ clearly must be $0$.

While the unitary operator $U(\tau)$ we obtained above depends on $\tau$, it is natural to ask whether it is possible for this example to establish the matrix equality in Corollary \ref{cor:2decomps} using a constant unitary matrix.  The answer is no. If we consider an arbitrary  matrix $U=\begin{bmatrix} a & b \\ c& d \end{bmatrix}$ that satisfies both $M^1_{\theta, \beta_2}(\tau)U=UM^1_{\theta, \widetilde{\beta_2}}(\tau)$ and $UU^*=I$, it is easy to see that the resulting system of equations forces $|b|=\frac{1}{|2-\tau|}$, so the matrix must vary with $\tau$.\end{example}


\section{Non-Uniqueness of $W(S_\theta^1)$}  \label{sec:W}

In this section, we investigate Question $2$, which asks if the equality $W(S^j_\theta)= W(S^j_\phi)$ for $j=1,2$ guarantees that RIFs $\theta$ and $\phi$ satisfy a nice relationship. Again, we will use the constructions from Section \ref{sec:product} to illustrate the challenges that arise. First, as noted in Theorem \ref{thm:unique1}, in the one-variable setting, $W(S_\theta) = W(S_\phi)$ if and only if $\theta = \lambda \phi$ for some $\lambda \in \mathbb{T}$. In the two-variable setting of $S^1_{\theta}, S^1_{\phi}$, there is an interesting analogue of the backwards direction of that statement (up to closures). Specifically, we have the following corollary of Theorem \ref{thm:uniqueness}: 

\begin{corollary} \label{cor:B1B2} Let $\theta, \phi$ be RIFs on $\mathbb{D}^2$. If there exist finite Blaschke products $B_1, B_2$ such that 
$B_1(z_2) \phi(z_1, z_2) = B_2(z_2) \theta(z_1, z_2),$ then $\text{Closure }(W(S^1_{\theta}))= \text{Closure }(W(S^1_{\phi})).$
\end{corollary} 

\begin{proof} First consider the case when $\deg_1(\phi)=0$. Then our assumptions imply that $\deg_1(\theta)=0$ as well. By our comments after the statement of Theorem \ref{thm:Mconstructive}, $W(S^1_{\theta}) =\emptyset = W(S^1_{\phi}).$ An analogous argument holds if $\deg_1(\theta)=0$.

Now assume $\deg_1(\phi), \deg_1(\theta) \ge 1$. Then \cite[Corollary 3.4]{bg17} states that for any such RIF $\psi$ (in particular, for $\theta$ and $\phi$), the closure of $W(S_{\psi}^1)$ is the convex hull of $\bigcup_{\tau \in \mathbb{T}} W(M^1_{\psi}(\tau))$. By Theorem \ref{thm:uniqueness}, 
each $M_\theta^1(\tau)$ is unitarily similar to $M_\phi^1(\tau)$ and hence, $W(M_\theta^1(\tau)) = W(M_\phi^1(\tau))$ for each $\tau \in \mathbb{T}$. Thus, the closure of  $W(S^1_{\theta})$ must equal the closure of $W(S^1_{\phi}).$ \end{proof}

Corollary \ref{cor:B1B2} only shows equality up to taking closures. However, Proposition \ref{prop:n1B} shows that if degree $(1,n)$ RIFs $\theta$ and $\phi$ only differ by an inner factor in $z_2$, then $W(S^1_{\theta})= W(S^1_{\phi}).$ Even in that degree $(1,n)$ case, two RIFs $\theta, \phi$ can yield equal numerical ranges without differing only by an inner factor in $z_2$. Consider the following example.

\begin{example} \label{ex:nrunique1} Let $p(z) = 2-z_1-z_2$ and $q(z) = 2- z_1 - z_2^2$ and set 
\[ \theta(z) = \frac{2 z_1 z_2-z_1-z_2}{2-z_1-z_2}  \ \ \text{ and } \ \  \phi(z) = \frac{2 z_1 z^2_2-z_1-z^2_2}{2-z_1-z^2_2}.\]
Then by Proposition \ref{prop:n1}, $S_\theta^1$ is unitarily equivalent to $T_{M_\theta^1}: H^2_2 \rightarrow H^2_2$ with $M_\theta^1(z_2) = \frac{1}{2-\bar{z}_2}$ and  $S_\phi^1$ is unitarily equivalent to $T_{M_\phi^1}: H^2_2 \rightarrow H^2_2$  with $M_\phi^1(z_2) = \frac{1}{2-\bar{z}_2^2}.$ By Corollary \ref{cor:n1nr},
\[ W(S_\theta^1) =  \text{the convex hull of }M_\theta^1(\mathbb{D}) \ \text{ and }  \  W(S_\phi^1)  = \text{the convex hull of } M_\phi^1(\mathbb{D}).\] It is easy to see that these two sets are exactly the same.
\end{example}

Example \ref{ex:nrunique1} suggests that one might try to prove that $W(S^1_\theta) = W(S^1_\phi)$ implies $\theta$, $\phi$ can only differ by either inner factors in $z_2$ or composition by finite Blaschke products in $z_2$.  However, it turns out that RIFs can yield the same numerical ranges but have quite different formulas. Specifically, in the following example, we construct two RIFs $\phi$ and $\psi$ so that $\deg \phi = (2,2) = \deg \psi$ and for given decompositions of $\mathcal{K}_\phi$ and $\mathcal{K}_{\psi}$, the associated numerical ranges in both variables are the same:
\[ W(S_\phi^1) = W(S_\psi^1) \text{ and } W(S_\phi^2) = W(S_\psi^2),\]
but there is no obvious connection between  $\phi$ and $\psi$.

\begin{example} \label{ex:nonuniquest} In this example, we consider two classes of RIFs $\phi_t$ and $\psi_s$, parameterized by $s,t \in (0,1)$, and show that for particular values of $s,t$, their associated numerical ranges coincide. First, let $\theta_t(z) = \frac{z_1 z_2 -t}{1-t z_1z_2}$ and define 
\begin{equation} \label{eqn:exfunction1} \phi_t(z) : = \theta_t(z)^2  = \left( \frac{z_1 z_2 -t}{1-t z_1z_2}\right)^2.\end{equation}
Then $\phi_t$ is of the form $B(z_1 z_2)$, where $B$ is a finite Blaschke product, and thus falls into the setting of Example \ref{example:1_1prod}. As described in that example, we can decompose 
$\mathcal{K}_{\phi_t} = \mathcal{S}_1 \oplus \mathcal{S}_2$, where each $\mathcal{S}_j$ is $S_{z_j}$-invariant and orthonormal bases $\beta_1$ and $\beta_2$ of $\mathcal{S}_1 \ominus z_1 \mathcal{S}_1$ and $\mathcal{S}_2 \ominus z_2 \mathcal{S}_2$ respectively are given below:
\[ \beta_1 =\left \{ \frac{z_1 \sqrt{1-t^2}}{1-t z_1 z_2}, \theta_t  \frac{z_1 \sqrt{1-t^2}}{1-t z_1 z_2} \right \}  \ \text{ and } \ \beta_2 = \left \{ \frac{ \sqrt{1-t^2}}{1-t z_1 z_2}, \theta_t  \frac{ \sqrt{1-t^2}}{1-t z_1 z_2} \right \}.\] 
Then, as in Example \ref{example:1_1prod}, the matrix-valued functions $M^1_{\phi_t}$ and  $M^2_{\phi_t}$ with respect to these bases are given by:

\[ M^1_{\phi_t}(z_2)  =  \begin{bmatrix} t & 0 \\ (1-t^2)  & t  \end{bmatrix} \bar{z}_2 \ \ \text{ and } \ \ M^2_{\phi_t}(z_1) =  \begin{bmatrix} t & 0 \\ (1-t^2)  & t  \end{bmatrix} \bar{z}_1.\]
By Theorem \ref{thm:Mconstructive}, we know $W(S^j_{\phi_t} )= W(T_{M^j_{\phi_t}})$ for $j=1,2$. By the symmetry between $M^1_{\phi_t} $ and $M^2_{\phi_t}$, we can see that 
$W(S^1_{\phi_t}) = W(S^2_{\phi_t})$ and so, we just need to compute $W(S^1_{\phi_t})$. For ease of computation, we first compute the closure of $W(S^1_{\phi_t})$.  To that end, fix $\tau \in \mathbb{T}$ and note that since $M^1_{\phi_t}(\tau)$ is a $2\times 2$ matrix, we can use the elliptical range theorem to conclude that $W (M^1_{\phi_t}(\tau))$ equals $\overline{D\big ( \bar{\tau}t, \tfrac{1-t^2}{2}\big)},$ the closed disk with center $\bar{\tau} t$ and radius $\frac{1-t^2}{2}$. Then by \cite[Corollary $3.4$]{bg17},
\[
\begin{aligned}
 \text{Closure  of } W(S^1_{\phi_t}) &= \text{Convex hull of } \bigcup_{\tau \in \mathbb{T}} W (M^1_{\phi_t}(\tau)) \\
 & = \text{Convex hull of } \bigcup_{\tau \in \mathbb{T}}\overline{D\big ( \bar{\tau} t, \tfrac{1-t^2}{2}\big)}\\
 & = \text{Convex hull of } \bigcup_{\tau \in \mathbb{T}} \bar{\tau} \overline{D\big (t, \tfrac{1-t^2}{2}\big)},
\end{aligned}
\]
where $\overline{D(t, \frac{1-t^2}{2})}$ is the closed disk centered at $t$ with radius $\frac{1-t^2}{2}$. Since the set $\cup_{\tau \in \mathbb{T}} \bar{\tau} \overline{D(t, \frac{1-t^2}{2})}$ is invariant under rotation about the origin and this property is preserved under taking convex hulls, we can conclude that the closure of $W(S^1_{\phi_t})$ is a closed disk centered at the origin. By considering $\tau =1$, we can see that the radius of this disk is $t + \frac{1-t^2}{2}$. 

Now, we claim that $W(S^1_{\phi_t}) = D(0, t + \frac{1-t^2}{2})$, the open disk centered at $0$ with radius $t + \frac{1-t^2}{2}$. Since $W(S^1_{\phi_t})$ and $D(0, t + \frac{1-t^2}{2})$ are convex sets with the equal closures and  $D(0, t + \frac{1-t^2}{2})$ is open, well-known properties of convex sets (see, for example, \cite[Corollary 6.3.2]{RT70}) imply that
\[ D\big (0, t + \tfrac{1-t^2}{2} \big) \subseteq W(S^1_{\phi_t}).\]
For the other containment, we show that if $x \in W(S^1_{\phi_t}) = W(T_{M^1_{\phi_t}})$, then $|x| < t + \frac{1-t^2}{2}.$ To that end, fix $x \in W(S^1_{\phi_t})$, so that there is an $\vec{f} =  \begin{bmatrix} f_1 \\ f_2  \end{bmatrix}\in H_2^2 \oplus H^2_2$  with $\| \vec{f} \|^2 := \|f_1 \|_{H^2_2}^2 + \|f_2\|_{H^2_2}^2 =1$ such that
\[ 
\begin{aligned}
x &= \langle T_{M^1_{\phi_t}} \vec{f}, \vec{f} \rangle_{H^2_2 \oplus H^2_2} = \left \langle   \begin{bmatrix} t & 0 \\ (1-t^2)  & t  \end{bmatrix} \begin{bmatrix} f_1 \\ f_2  \end{bmatrix},  \begin{bmatrix} z_2 f_1 \\ z_2 f_2  \end{bmatrix} \right \rangle_{H^2_2 \oplus H^2_2} \\
& = t \langle f_1, z_2 f_1 \rangle_{H_2^2} + (1-t^2) \langle f_1, z_2 f_2 \rangle_{H^2_2} +t \langle f_2, z_2 f_2\rangle_{H^2_2}.
\end{aligned}
\]
If $f_1$ or $f_2$ is the zero function, then $| x | \le t$ and so $x \in D(0, t + \frac{1-t^2}{2})$. Otherwise, recall that the numerical range of $S_z$ on $H^2(\mathbb{D})$ is $\mathbb{D}$, see \cite{klein}. Then
$W(S_{z_2}|_{H^2_2}) = \mathbb{D}$ and since $t >0$, we have
\[ 
\begin{aligned}
|x| &\le t \| f_1 \|^2 \left | \left \langle \frac{f_1}{\|f _1\|}, z_2 \frac{f_1}{\|f _1\|} \right \rangle_{H^2_2} \right | + (1-t^2) \|f_1 \|  \| f_2\| + t \| f_2 \|^2 \left | \left \langle \frac{f_2}{\|f _2\|}, z_2 \frac{f_2}{\|f _2\|} \right \rangle_{H^2_2}\right |  \\
&< t \| f_1\|^2 + (1-t^2) \| f_1\| \sqrt{ 1- \|f_1\|^2} + t \| f_2\|^2  \\
& \le t +\tfrac{(1-t^2)}{2},
\end{aligned} 
\] 
as needed.  This means $W(S^1_{\phi_t})$ (and hence, $W(S^2_{\phi_t})$) equals  the open disk $D(0, t + \frac{1-t^2}{2})$.

For our second collection of functions, fix $s \in[0,1)$ and consider 
\begin{equation} \label{eqn:exfunction2} \psi_s(z) : = z_1 z_2\theta_s(z)  = z_1z_2\left( \frac{z_1 z_2 -s}{1-s z_1z_2}\right).\end{equation}
Again, this function is of the form $B(z_1 z_2)$, for $B$ a finite Blaschke product and thus, we can apply Example \ref{example:1_1prod} using $\theta_s$ as the first factor and $z_1z_2$ as the second. This allows us to decompose
$\mathcal{K}_{\psi_s} =  \widetilde{\mathcal{S}}_1\oplus   \widetilde{\mathcal{S}}_2$, where each $\widetilde{\mathcal{S}}_j$ is $S_{z_j}$-invariant and orthonormal bases $\widetilde{\beta}_1$ and $\widetilde{\beta}_2$ of $ \widetilde{\mathcal{S}}_1 \ominus z_1 \widetilde{\mathcal{S}}_1$ and $ \widetilde{\mathcal{S}}_2 \ominus z_2  \widetilde{\mathcal{S}}_2$ respectively are given below:
\[ \widetilde{\beta}_1 =\left \{ \frac{z_1 \sqrt{1-s^2}}{1-s z_1 z_2}, \theta_s z_1 \right \}  \ \text{ and } \ \widetilde{\beta}_2 = \left \{ \frac{ \sqrt{1-s^2}}{1-s z_1 z_2}, \theta_s \right \}.\] 
Then, as in Example \ref{example:1_1prod}, the matrix-valued functions $M^1_{\psi_s}$ and  $M^2_{\psi_s}$ with respect to these bases are given by:
\[ M^1_{\psi_s} =  \begin{bmatrix} s & 0 \\ \sqrt{1-s^2}  & 0  \end{bmatrix} \bar{z}_2 \ \ \text{ and } \ \ M^2_{\psi_s} =  \begin{bmatrix} s & 0 \\ \sqrt{1-s^2}  & 0  \end{bmatrix} \bar{z}_1.\]
Thus, we have $W(S_{\psi_s}^1) = W(S_{\psi_s}^2)$. To compute $W(S_{\psi_s}^1)$, fix  $\vec{f} =  \begin{bmatrix} f_1 \\ f_2  \end{bmatrix}\in H_2^2 \oplus H_2^2$  with $\| \vec{f} \|^2 = \|f_1 \|_{H^2_2}^2 + \|f_2\|_{H^2_2}^2 =1$. Then 
\[ \begin{aligned}  \langle T_{M^1_{\psi_s}} \vec{f}, \vec{f} \rangle_{H_2^2 \oplus H_2^2}  &= \left \langle   \begin{bmatrix} s & 0 \\  \sqrt{1-s^2}  & 0  \end{bmatrix} \begin{bmatrix} f_1 \\ f_2  \end{bmatrix},  \begin{bmatrix} z_2 f_1 \\ z_2 f_2  \end{bmatrix} \right \rangle_{H_2^2 \oplus H_2^2} \\
&= s \langle f_1, z_2 f_1 \rangle_{H^2_2} + \sqrt{1-s^2} \langle f_1, z_2 f_2 \rangle_{H^2_2} \\
& = \sqrt{1-s^2}  \| f_1 \| \| f_2 \| w +s \| f_1\|^2 \zeta  \\
& = xw\sqrt{1-s^2}\sqrt{1-x^2} +s x^2 \zeta,
\end{aligned}
\]
where $x = \| f_1 \|_{H_2^2} \in [0,1]$ and (with a slight adjustment  of the formulas if $f_1$ or $f_2$ are equivalently zero),
\[   w = \frac{\langle f_1, z_2 f_2 \rangle_{H_2^2} }{ \| f_1 \|_{H_2^2}  \| f_2 \|_{H_2^2}}  \in \overline{\mathbb{D}} \ \text{ and } \ \zeta = \left \langle \frac{f_1}{\| f_1 \|_{H_2^2}}, z_2\frac{f_1}{\| f_1 \|_{H_2^2}} \right \rangle_{H_2^2} \in W(S_{z_2}|_{H^2_2})= \mathbb{D}.\]
 If $f_1\equiv 0$, then $ \langle T_{M^1_{\psi_s}} \vec{f}, \vec{f} \rangle =0$. If $f_2 \equiv 0$ so $\|f_1 \| =1$,  then $ \langle T_{M^1_{\psi_s}} \vec{f}, \vec{f} \rangle =s x^2 \zeta$ and ranging over all such $f_1$ causes $\zeta$ to range over  all of $\mathbb{D}$. Meanwhile, allowing $f_1$ and $f_2$ to range over all pairs of functions with both $f_1, f_2 \not \equiv 0$ in $H_2^2$ and $\| \vec{f}\|=1$ exactly causes $w$ and $\zeta$ to range over $\overline{\mathbb{D}}$ and $\mathbb{D}$ respectively and $x$ to cover the range $(0,1)$. Thus, we have 
\[ W(T_{M^1_{\psi_s}}) = \left \{ xw\sqrt{1-s^2}\sqrt{1-x^2} +s x^2 \zeta: x \in [0,1], w \in \overline{\mathbb{D}}, \zeta \in \mathbb{D}. \right\}.\]
This characterization implies that if $\alpha \in W(T_{M^1_{\psi_s}})$, then $e^{i r} \alpha \in W(T_{M^1_{\psi_s}})$ for all $r\in [0, 2\pi)$. Thus, $W(T_{M^1_{\psi_s}})$ is a disk centered at the origin. The radius of that circle is 
\[
\begin{aligned}
 w(T_{M^1_{\psi_s}}) &= \sup \left \{ |\alpha |: \alpha \in  W(T_{M^1_{\psi_s}}) \right \} \\
&= \max \left \{ x\sqrt{1-s^2}\sqrt{1-x^2} +s x^2 : x \in [0,1] \right \} \\
& = \tfrac{ s+1}{2},
\end{aligned}
\]
where the last step is a calculus computation.
Because we only allow $\zeta \in \mathbb{D}$ and $s >0$, one can see that the numerical radius is never actually attained by a point in the numerical range. Thus $W(S^1_{\psi_s})=W(T_{M^1_{\psi_s}})$ is the open disk centered at the origin with radius $\frac{s+1}{2}$. 

To summarize, for all $s, t \in (0, 1)$, we have 
\[ \begin{aligned}
W(S^1_{\phi_t}) &= W(S^2_{\phi_t})= D\big (0, t + \tfrac{1-t^2}{2} \big), \\
W(S^1_{\psi_s}) &= W(S^2_{\psi_s}) = D\big(0, \tfrac{s+1}{2} \big). 
\end{aligned}
\]
Fix $t \in (0,1)$. Setting $t + \tfrac{1-t^2}{2} = \tfrac{s+1}{2}$ gives $s= t (2-t) \in (0,1)$. Thus, the functions 
\[ \phi_t(z) =  \left( \frac{z_1 z_2 -t}{1-t z_1z_2}\right)^2 \ \ \text{ and } \  \ \psi_{t(2-t)}(z) =  z_1z_2\left( \frac{z_1 z_2 - t(2-t)}{1-t(2-t) z_1z_2}\right) \]
have degree $(2,2)$ and identical numerical ranges $W(S^1_{\phi_t}) = W(S^1_{\psi_{t(2-t)}})$ and  $W(S^2_{\phi_t}) = W(S^2_{\psi_{t(2-t)}})$, but there is no obvious relationship (at least to us) between the two functions.
\end{example}


\section{Open and closed numerical ranges} \label{sec:open}

In this section, we investigate Question $3$, which asks when the numerical ranges $W(S_\theta^j)$ are open and when they are closed. In Remark \ref{rem:open}, we saw that if $\theta$ has degree $(1,n)$, then there are two cases: if  $\theta$ is the product of a degree $1$ Blaschke product in $z_1$ and a degree $n$ Blaschke product in $z_2$, then $W(S_\theta^1)$ is closed. Otherwise,  $W(S_\theta^1)$ is open.

 This suggests the following somewhat speculative conjecture:

\begin{conjecture}  \label{conj:open} Let $\theta =\lambda \frac{\tilde{p}}{p}$ be an RIF on $\mathbb{D}^2$. If $\theta$ does not factor as $B_1(z_1) B_2(z_2)$, for finite Blaschke products $B_1$ and $B_2$, then $W(S_\theta^1)$ and $W(S_\theta^2)$ are open.  If $\theta$ does factor as $B_1(z_1) B_2(z_2)$, then  $W(S_\theta^1)$ and $W(S_\theta^2)$ are closed. 
\end{conjecture}

As shown in the following proposition, the factorization condition in Conjecture \ref{conj:open} implies that there is at least one decomposition of $\mathcal{K}_\theta$ where the resulting $W(S_\theta^1)$ is closed. 

\begin{proposition} \label{prop:open} Let $\theta =  B_1(z_1) B_2(z_2)$ for finite Blaschke products $B_1, B_2$. Then there is a decomposition of $\mathcal{K}_\theta =\mathcal{S}_1 \oplus \mathcal{S}_2$ with $W(S_\theta^1)$ closed. 
\end{proposition}
\begin{proof} If $\deg_1(\theta)=0$, then our remarks after Theorem \ref{thm:Mconstructive} imply that $W(S_\theta^1) =\emptyset$ and thus, is closed. 

Now assume $\deg_1(\theta) \ge 1$. By examining reproducing kernels, we can decompose $\mathcal{K}_{\theta} = \mathcal{K}_{B_1} + B_1  \mathcal{K}_{B_2}$ and define
\[
\begin{aligned}
\mathcal{S}_1 &:=   B_1  \mathcal{K}_{B_2} = \mathcal{H}\left( \tfrac{B_1(z_1) \overline{B_1(w_1)}(1-B_2(z_2) \overline{B_2(z_2)})}{(1-z_1 \bar{w}_1)(1-z_2 \bar{w}_2)} \right),\\
\mathcal{S}_2 &:= \mathcal{K}_{B_1} = \mathcal{H}\left( \tfrac{1-B_1(z_1) \overline{B_1(w_1)}}{(1-z_1 \bar{w}_1)(1-z_2 \bar{w}_2)} \right).
\end{aligned}
\]
Since $\mathcal{S}_1  \cap \mathcal{S}_2 = \{0\}$, the discussion in Section \ref{sec:2var} implies that we have an orthogonal decomposition $\mathcal{K}_{\theta} =\mathcal{S}_1 \oplus \mathcal{S}_2$ with each $\mathcal{S}_j$ being $z_j$-invariant and   $\mathcal{S}_2 \ominus z_2 \mathcal{S}_2 =  \mathcal{H}\left( \frac{1-B_1(z_1) \overline{B_1(z)}}{1-z_1 \bar{w}_1} \right),$ the one-variable model space associated to $B_1$.

As in \cite[Example 3.3]{bg17}, we can choose a nice basis for this one-variable model space and obtain a matrix symbol $M^1_\theta$ that agrees with the constant matrix $M_{B_1}$ given in \eqref{eqn:SB}. Since $M_{B_1}$ is a constant matrix, $W(T_{M_{\theta}^1}) = W(T_{M_{B_1}}) = W(M_{B_1})$ is closed and hence, $W(S^1_\theta)$ is closed. 
\end{proof}

Unfortunately, the factorization condition in Conjecture \ref{conj:open} does not imply that the $M^1_\theta$ associated with every $\mathcal{K}_\theta$ decomposition is constant. For example, consider the following:

\begin{example}\label{ex:open1} Let $\theta(z) = z_1 z_2$ and $\Theta(z) = z_1 \theta(z)$. Then by Lemma \ref{lemma:1_1}, we can decompose $\mathcal{K}_\theta = \mathcal{S}_1^\theta \oplus \mathcal{S}_2^\theta$ where each $\mathcal{S}_j^\theta$ is $S_{z_j}$-invariant and $\{ z_2\}$ is an orthonormal basis of $\mathcal{S}_2^\theta \ominus z_2 \mathcal{S}_2^\theta$. 
Given that,  Example \ref{example:z_1product} implies that we can decompose $\mathcal{K}_\Theta = \mathcal{S}_1 \oplus \mathcal{S}_2$, where each $\mathcal{S}_j$ is $S_{z_j}$-invariant and $\{ 1, z_1 z_2\}$ is an orthonormal basis of $\mathcal{S}_2 \ominus z_2 \mathcal{S}_2$. 
Then the formula in Example \ref{example:z_1product} (or direct computation) gives the associated matrix $M_\Theta^1$: 
\[M^1_{\Theta}=\begin{bmatrix}0 & 0 \\ \bar{z}_2 & 0 \end{bmatrix}.\]
Although $M^1_{\Theta}$ is not constant, one can show that the numerical range of $T_{M^1_\Theta}: H^2_2 \rightarrow H^2_2$ is the closed Euclidean disk centered at the origin with radius $\tfrac{1}{2}$. This is implied by \cite[Theorem 2]{bs}, but we can also prove it directly. To identify the elements in $W(T_{M^1_{\Theta}})$, fix any
\[ \vec{f} = \begin{bmatrix} f_1 \\ f_2 \end{bmatrix}  \text{ with each } f_1, f_2 \in H^2_2\ \text{ and } \ \| f_1\|_{H^2_2}^2 + \|f_2\|_{H^2_2}^2 =1.\]
Then 
\begin{equation} \label{eqn:ex7:nr}
\left \langle T_{M^1_{\Theta}} \vec{f} , \vec{f} \right \rangle_{H^2_2 \oplus H^2_2} =   \left \langle f_1, z_2f_2 \right \rangle_{H^2_2}.
\end{equation}
By replacing $f_1$ with $e^{i\alpha} f_1$, one can check that the set of points in \eqref{eqn:ex7:nr} is closed under rotation. Thus, $W(T_{M^1_{\Theta}})$ must be a disk centered at the origin. Moreover each point satisfies
\[ \left | \left\langle T_{M^1_{\Theta}} \vec{f} , \vec{f} \right \rangle_{H^2_2 \oplus H^2_2} \right|  =\left | \left \langle f_1, z_2f_2 \right \rangle_{H^2_2} \right | \le \| f_1 \|_{H^2_2}\cdot \| f_2 \|_{H^2_2} \le \tfrac{1}{2}.\]
Since $\tfrac{1}{2}$ can be obtained by setting $f_1 = \frac{1}{\sqrt{2}} z_2$, $f_2 = \frac{1}{\sqrt{2}}$, we can see that  $W(T_{M^1_{\Theta}})$  is the closed Euclidean disk centered at the origin with radius $\frac{1}{2}$. Thus, $W(S^1_{\Theta})$ is closed even though $M^1_\Theta$ is nonconstant. This example supports Conjecture \ref{conj:open} but shows that proving it might be challenging. 
\end{example} 

A full answer to Conjecture \ref{conj:open} is beyond the scope of this paper.  However, we do have the following partial result, which uses some ideas from \cite{klein}. 

\begin{theorem}  \label{thm:open} Let $\theta =\lambda \frac{\tilde{p}}{p}$ be an RIF on $\mathbb{D}^2$ with $\deg_1(\theta) \ge 1$. If there is no single straight line that intersects the numerical ranges $W(M^1_{\theta}(\tau))$ for all $\tau \in \mathbb{T}$, then $W(S_\theta^1)$ must be open.
\end{theorem}

\begin{proof} We prove this via contrapositive. Specifically, we will assume $W(S_\theta^1) = W(T_{M^1_\theta})$  is not open and show that there are $\beta, \alpha \in \mathbb{R}$ such that $\beta \in \text{Im} \left (e^{i\alpha} W(M^1_\theta(\tau)\right))$ for each $\tau \in \mathbb{T}$.  Then $y = \beta$ intersects each set $e^{i\alpha} W(M^1_\theta(\tau))$. Equivalently, rotating $y=\beta$ by $-\alpha$ radians implies that that line $\cos(\alpha) y + \sin(\alpha) x = \beta$ intersects each $M^1_\theta(\tau)$.

Assume that $W(T_{M^1_\theta})$  is not open and let $b$ be a point that is in $W(T_{M^1_\theta}) \cap \partial W(T_{M^1_\theta})$. Let $m$ be the dimension of $M^1_\theta$ and let $I_m$ denote the $m\times m$ identity matrix. Then $ 0 \in W(T_{M^1_\theta - bI_m}) \cap \partial W(T_{M^1_\theta - bI_m})$. Because $W(T_{M^1_\theta - bI_m})$ is convex, there is an angle $\alpha$ so that if we rotate $W(T_{M^1_\theta - bI_m})$ by $\alpha$, the new set is fully contained in closed upper-half-plane $ \{ x+iy \in \mathbb{C}: y \ge 0\}$. Equivalently, 
\[ \text{Im}\left(W(T_{ e^{i \alpha} M^1_\theta -  e^{i \alpha} bI_m})\right) = \text{Im} \left( e^{i \alpha} W(T_{M^1_\theta - bI_m})\right) \subseteq [0, \infty)\]
and contains $0$. For simplicity, define the function
\[F_\theta: = e^{i \alpha} M^1_\theta -  e^{i \alpha} bI_m,\]
a continuous matrix-valued function on $\overline{\mathbb{D}}$. Then $\text{Im}\left(W(T_{F_\theta})\right)  \subseteq [0, \infty)$. Thus, for all $\vec{g} \in H^2_2(\mathbb{D})^m$, 
\[ 
\begin{aligned}
0 \le \text{Im} \left( \left \langle T_{F_\theta} \vec{g}, \vec{g} \right \rangle_{H^2_2(\mathbb{D})^m} \right)  
 &= \int_{\mathbb{T}} \text{Im}\left( \left \langle F_\theta(\tau) \vec{g}(\tau), \vec{g}(\tau) \right \rangle_{\mathbb{C}^m} \right) \ d m (\tau) \\
 & = \int_{\mathbb{T}}\left \langle  \text{Im}\left( F_\theta(\tau)\right) \vec{g}(\tau), \vec{g}(\tau) \right \rangle_{\mathbb{C}^m} \ d m (\tau),
\end{aligned}
\]
where $dm$ is normalized Lebesgue measure on $\mathbb{T}$ and $\text{Im}\left( F_\theta(\tau)\right) = \frac{1}{2i} \left( F_\theta(\tau) - F_\theta(\tau)^*\right).$ We claim that the above inequality implies that $ \text{Im}\left( F_\theta (\tau)\right)$ must be a positive semi-definite matrix for every $\tau \in \mathbb{T}$. Otherwise,  if there was a $\tau_0 \in \mathbb{T}$ and a vector $\vec{x} \in \mathbb{C}^m$ such that 
\[  \left \langle  \text{Im}\left( F_\theta(\tau_0)\right) \vec{x}, \vec{x} \right \rangle_{\mathbb{C}^m} <0,\]
then we could find a vector-valued function $\vec{G} \in H^2_2(\mathbb{D})^m$ such that 
\[  \int_{\mathbb{T}}\left \langle  \text{Im}\left( F_\theta(\tau)\right) \vec{G}(\tau), \vec{G}(\tau) \right \rangle_{\mathbb{C}^m} \ d m (\tau) <0,\]
a contradiction. For example, one could take $\vec{G}(z) = \vec{x} \frac{1}{z - (1+ \epsilon) \tau_0}$ for sufficiently small $\epsilon>0$.

Thus,  $ \text{Im}\left( F_\theta (\tau)\right)$ must be a positive semi-definite matrix for every $\tau \in \mathbb{T}$. Recall that $0 \in W(T_{F_\theta})$. Thus, there exists $\vec{g}_0 \in H^2_2(\mathbb{D})^m$ with $\| \vec{g}_0 \| =1$ and 
\[ 0 = \text{Im}\left(  \left \langle T_{F_\theta} \vec{g}_0, \vec{g}_0 \right \rangle_{H^2_2(\mathbb{D})^m} \right)=  \int_{\mathbb{T}}\left \langle  \text{Im}\left( F_\theta(\tau)\right) \vec{g}_0(\tau), \vec{g}_0(\tau) \right \rangle_{\mathbb{C}^m} \ d m(\tau). \]
Our earlier arguments imply that 
\[ \left \langle \text{Im}\left( F_\theta(\tau)\right) \vec{g}_0(\tau), \vec{g}_0(\tau) \right \rangle_{\mathbb{C}^m} \ge 0\] 
for all $\tau \in \mathbb{T}$ where $g_0$ is defined and so, it must be the case that 
\[ \left \langle \text{Im}\left( F_\theta(\tau)\right) \vec{g}_0(\tau), \vec{g}_0(\tau) \right \rangle_{\mathbb{C}^m} = 0\] 
for almost every $\tau \in \mathbb{T}$.  Because $\vec{g}_0$ is an element in $H^2_2(\mathbb{D})^m$  that is not identically $\vec{0}$, its boundary values cannot equal $\vec{0}$ on a set of positive measure in $\mathbb{T}$. This implies that 
\[ 0 \in W(\text{Im} ( F_\theta(\tau)) ),\]
for almost every $\tau \in \mathbb{T}$ and since $F_\theta$ is continuous on $\mathbb{T}$, this implies that $0 \in W(\text{Im} ( F_\theta(\tau)))$ for every $\tau \in \mathbb{T}$. Then, using properties of the numerical range and the definition of $F_\theta$, we have 
\[ 0 \in W(\text{Im} ( F_\theta(\tau)))  = \text{Im}\left( W( e^{i \alpha} M^1_\theta -  e^{i \alpha} bI_m)\right) = \text{Im} \left(e^{i\alpha} W(M^1_\theta(\tau))\right) - \text{Im}\left(e^{i\alpha} b\right).\]
Setting $\beta =    \text{Im}\left(e^{i\alpha} b\right)$, this implies $\beta \in \text{Im} \left(e^{i\alpha} W(M^1_\theta(\tau)\right)$ for each $\tau \in \mathbb{T}$ as needed. Combined with our earlier analysis, this proves the claim. 
\end{proof}

We illustrate this theorem with an example:

\begin{example} \label{ex:open2}
Fix $t \in (0,1).$ Let $\theta_t(z) = \frac{z_1 z_2 - t}{1 - t z_1 z_2}$ and define $$\phi_t(z):=\theta_t(z)^2=\bigg(\frac{z_1 z_2 - t}{1 - t z_1 z_2}\bigg)^2$$ an RIF as in Example \ref{ex:nonuniquest}. We will use Theorem \ref{thm:open} to show that for $t>\frac{\sqrt{5}-1}{2},$ $W(S_\theta^1)$ must be open.  
Specifically, we will show that for such $t$-values, there is no single straight line that intersects all numerical ranges $W(M^1_{\theta}(\tau))$ for $\tau=1, -1, i.$ By Example \ref{ex:nonuniquest}, each $W (M^1_{\phi_t}(\tau))=\bar{\tau} \overline{D\big (t, \tfrac{1-t^2}{2}\big)},$ the closed disk centered at $\bar{\tau}t$ and radius $\frac{1-t^2}{2},$ and so equivalently, we will show that there is no single straight line that intersects all of 
$$D_1:=\overline{D_1 (t, r)}, \; D_2:=\overline{D_2 (-t, r)}\;\; \text{and} \;\;  D_3:= \overline{D_3 (-it, r)},$$ where $r=\frac{1-t^2}{2}.$ Let $p_1$ be a point on the boundary of the disk $D_1$ given by $$p_1=(t, 0)+r(\cos\delta_1, \sin\delta_1)=(t+r\cos\delta_1, \sin\delta_1),$$ $p_2$ be a point on the boundary of the disk $D_2$ given by $$p_2=(-t, 0)+r(\cos\gamma_1, \sin\gamma_1)=(-t+r\cos\gamma_1, r\sin\gamma_1)$$ and $p_3$ a point on the boundary of the disk $D_3$ given by $$p_3=(0, -t)+r(\cos\gamma_2, \sin\gamma_2)=(r\cos\gamma_2,-t+r\sin\gamma_2 ).$$ The slope from $D_2$ to $D_3$ is given by $$s_{23}=\frac{-t+r(\sin\gamma_2-\sin\gamma_1)}{t+r(\cos\gamma_2-\cos\gamma_1)},$$ where $$-t+r(\sin\gamma_2-\sin\gamma_1)\in [-t-2r, -t+2r]\; \text{and}\; t+r(\cos\gamma_2-\cos\gamma_1)\in [t-2r, t+2r].$$ If $t>\frac{\sqrt{5}-1}{2},$ then $t^2+t-1=\left(t-\frac{\sqrt{5}-1}{2}\right)\left(t+\frac{\sqrt{5}+1}{2}\right)>0.$ Then $-t+2r=-t^2-t+1<0$ and $t-2r=t^2+t-1>0,$ which implies $s_{23}<0.$ Similarly, we obtain a positive slope $$s_{13}=\frac{-t+r(\sin\delta_2-\sin\delta_1)}{-t+r(\cos\delta_2-\cos\delta_1)}$$ from $D_1$ to $D_3$ for $t>\frac{\sqrt{5}-1}{2}.$ Therefore, for $t>\frac{\sqrt{5}-1}{2},$ the lines from $D_2$ to $D_3$ have a negative slope, while those from $D_1$ to $D_3$ have a positive slope. This forces a contradiction, that is, for $t>\frac{\sqrt{5}-1}{2},$ no single line can intersect all three disks arranged as stated and by Theorem \ref{thm:open}, $W(S_\theta^1)$ must be open.
\end{example}

\end{document}